\documentclass{article}

\usepackage{amsmath}
\usepackage{amsfonts}

\usepackage{lipsum}
\usepackage{amsfonts}
\usepackage{graphicx}
\usepackage{epstopdf}
\usepackage{algorithmic}
\usepackage{todonotes}
\usepackage{enumitem}
\usepackage{subcaption}
\usepackage{comment}
\usepackage{url}

\usepackage{amssymb}
\usepackage{amsthm}
\usepackage{xfrac}

\newtheorem{theorem}{Theorem}
\newtheorem{lemma}{Lemma}

\newtheorem{corollary}{Corollary}
\newtheorem{proposition}{Proposition}

\DeclareMathOperator{\Tr}{Tr}
\DeclareMathOperator{\In}{In}

\newtheorem{assumption}[theorem]{Assumption}

\author{Susanne Bradley and Chen Greif}
\title{Eigenvalue Bounds for Double Saddle-Point Systems}

\author{Susanne Bradley\thanks{Department of Computer Science, The University of British Columbia, Vancouver, Canada V6T 1Z4
		({smbrad@cs.ubc.ca}, {greif@cs.ubc.ca}).}
	\and Chen Greif\footnotemark[2]}

\date{October 25, 2021}

\usepackage{amsopn}

\newcommand{\sK}{\mathcal{K}}
\newcommand{\sM}{\mathcal{M}}

\begin{document}

	\maketitle
	
	\begin{abstract}
		We derive bounds on the eigenvalues of a generic form of double saddle-point matrices. The bounds are expressed in terms of extremal eigenvalues and singular values of the associated block matrices. Inertia and algebraic multiplicity of eigenvalues are considered as well. The analysis  includes bounds for preconditioned matrices based on block diagonal preconditioners using Schur complements, and it is shown that in this case the eigenvalues are clustered within a few intervals bounded away from zero.  Analysis for approximations of Schur complements is included.  Some numerical experiments validate our analytical findings. 
	\end{abstract}

\section{Introduction}
Given positive integer dimensions $n \ge m \ge p$, consider the $(n+m+p) \times (n+m+p)$ double saddle-point system
\begin{equation}
	\sK  u = b,
	\label{eq:sp_system}
\end{equation}
where 
\begin{equation}
	\label{eqn:2spdef}
	\sK = \begin{bmatrix}
		A & B^T & 0 \\
		B & -D & C^T \\
		0 & C & E
	\end{bmatrix} ; \quad 
	u = \begin{bmatrix}
		x \\
		y \\
		z
	\end{bmatrix}; \quad
	b = \begin{bmatrix}
		p \\
		q \\
		r
	\end{bmatrix}.
\end{equation}
In \eqref{eqn:2spdef}, $A\in \mathbb{R}^{n \times n}$ is assumed symmetric positive definite, $D \in \mathbb{R}^{m \times m}$ and $E \in \mathbb{R}^{p \times p}$ are positive semidefinite, and $B \in \mathbb{R}^{m \times n}$, $C \in \mathbb{R}^{p \times m}.$

Under the assumptions stated above, $\sK$  is symmetric and indefinite, and solving \eqref{eq:sp_system} presents several numerical challenges. When the linear system is too large  for direct solvers to work effectively, iterative methods \cite{s2003} are preferred; this is the scenario that is of interest to us in this work. A minimum residual Krylov subspace solver such as MINRES \cite{ps1975} is a popular choice due to its optimality and short recurrence properties. Other solvers may be highly effective as well. 

Linear systems of the form \eqref{eq:sp_system}--\eqref{eqn:2spdef} appear frequently in multiphysics problems, and their numerical solution is of increasing importance and interest. There is a large number of relevant applications \cite{cai21, multiphysics2013, sogn18}; liquid crystal problems \cite{beik18, rg2013}, Darcy-Stokes equations \cite{cai09, mardal2021}, coupled poromechanical equations \cite{ferronato19}, magma-mantle dynamics \cite{rhebergen15}, and PDE-constrained optimization problems \cite{pearson12, rees10b} are just a small subset of linear systems that fit into this framework.

The matrices $D$ and $E$ are often associated with regularization (or stabilization). We refer to the special case of $D=E=0$ as an {\em unregularized form} of $\sK$, and denote it by $\sK_0$:
\begin{equation}
	\sK_0 = \begin{bmatrix}
		A & B^T & 0 \\
		B & 0& C^T \\
		0 & C & 0
	\end{bmatrix}.
	\label{eq:sK0}
\end{equation}
This simpler form is of much potential interest, as it may be considered a direct generalization of the standard saddle-point form for $2 \times 2$ block matrices:
\begin{equation}
	\tilde{\sK}_0 = \begin{bmatrix}
		A & B^T  \\
		B & 0 
	\end{bmatrix}.
	\label{eq:K2by2}
\end{equation}

Matrices of the form~\eqref{eq:K2by2} have been extensively studied, and their analytical and numerical properties are well understood; see \cite{bgl05, miro2018} for excellent surveys. Some properties of $\sK$ follow from appropriately reordering and partitioning the block matrix and then using known results for block-$2 \times 2$ saddle-point matrices (as given in, for example, \cite{gould09, ruiz18, rusten92, silvester94}).
Specifically, $\sK$ can be reordered and partitioned into a $2 \times 2$ block matrix
\begin{equation}
	\label{eqn:k_2}
	\tilde{\sK} = \left[\begin{array}{c c | c}
		A & 0 & B^T \\
		0 & E & C \\
		\hline
		B & C^T & -D
	\end{array}
	\right].
\end{equation}
While this approach is often effective, we may benefit from considering the block-$3 \times 3$ formulation $\sK$ directly, without resorting to \eqref{eqn:k_2}. When $E$ is rank-deficient or zero (as occurs, for example, in \cite{gatica00,langer07}), the leading $2 \times 2$ block of $\tilde{\sK}$ is singular, even if $A$ is full rank. It is then more challenging to develop preconditioners for $\tilde{\sK}$ or to obtain bounds on its eigenvalues, as we are restricted to methods that can handle singular leading blocks. Additionally, by considering the block-$3 \times 3$ formulation in deriving eigenvalue bounds, we will see in this paper that we can derive effective bounds using the singular values of $B$ and $C$ (along with the eigenvalues of the diagonal blocks). Analyzing $\tilde{\sK}$ using established results for block-$2 \times 2$ matrices (such as those given by Rusten and Winther \cite{rusten92}) instead requires singular values of the larger off-diagonal block $\begin{bmatrix} B & C^T \end{bmatrix}$, which may be more difficult to obtain than singular values of $B$ and $C$. (We can estimate the singular values of $\begin{bmatrix} B & C^T \end{bmatrix}$ in terms of those of $B$ and $C$, but if these estimates are loose it will result in loose eigenvalue bounds.) 

There has been a recent surge of interest in the iterative solution of multiple saddle-point problems, and our work adds to the increasing body of literature that considers these problems. Recent papers that provide interesting analysis are, for example, \cite{beik18, beigl20, cai21, pearson21, sogn18}. 

The distribution of eigenvalues of $\sK$ plays a central role in determining the efficiency of iterative solvers. It is therefore useful to gain an understanding of the spectral structure as part of the selection and employment of solvers. Effective preconditioners are instrumental in accelerating convergence for sparse and large linear systems; see \cite{b2002, pp2020, s2003, w2015} for general overviews, and \cite{bgl05, pw2015, miro2018} for a useful overview of  solvers and preconditioners  for saddle-point problems. For multiphysics problems it has been demonstrated that exploiting the properties of the underlying discrete differential operators and other characteristics of the problem at hand is beneficial in the development of robust and fast solvers; see, e.g., \cite{ferronato19}.

There are several potential approaches to the development of preconditioners for the specific class of problem that we are considering. Monolithic preconditioners, which work on the entire matrix, have been recently shown to be extremely effective. Recent work such as \cite{monolithic} has pushed the envelope towards scalable solvers based on this methodology.  Another increasingly important approach is operator preconditioning based on continuous spaces; see \cite{h2006, mw2011}. This approach relies on the properties of the underlying continuous differential operators, and uses tools such as Riesz representation and natural norm considerations to derive block diagonal preconditioners. Block diagonal preconditioners can also be derived directly by linear algebra considerations accompanied by properties of discretized PDEs; see, for example, \cite{elman05, pw2015}.

The preconditioner we consider for our analysis is:
\begin{equation}
	\label{eq:M}
	\sM := \begin{bmatrix}
		A & 0 & 0 \\
		0 & S_1 & 0 \\
		0 & 0 & S_2
	\end{bmatrix},
\end{equation}
where 
\begin{equation}
	S_1 = D+BA^{-1}B^T ; \quad
	S_2 = E + CS_1^{-1}C^T.
	\label{eq:S12}
\end{equation} 
We assume that $S_1$ and $S_2$ are both positive definite. 

The preconditioner $\sM$ is based on Schur complements. It has been considered in, for example, \cite{cai21, sogn18}, and is a natural extension of \cite{i2001, mgw2000} for block-$2 \times 2$ matrices of the form \eqref{eq:K2by2}. In practice the Schur complements $S_1$ and $S_2$ defined in~\eqref{eq:S12} are too expensive to form and invert exactly. It is therefore useful to consider approximations to those matrices when a practical preconditioner is to be developed, and we include an analysis of that scenario.

The recent papers of Sogn and Zulehner \cite{sogn18} and Cai et al. \cite{cai21} both analyze the performance of a block-$n \times n$ block diagonal preconditioner analogous to $\sM$ defined in \eqref{eq:M}, though the former focus on the spectral properties of the continuous (rather than discretized) preconditioned operator, and Cai et al. focus their analyses of this preconditioner on the case where all diagonal blocks except $A$ are zero. Existing analyses of the eigenvalues of unpreconditioned block-$3 \times 3$ matrices have often been restricted to specific problems, such as interior-point methods in constrained optimization; see \cite{gmo14, morini16}.

Our goal in this paper is to provide a general framework for analysis of eigenvalue bounds  for the double saddle-point case, with a rather minimal set of assumptions on the matrices involved. To our knowledge, ours is the first work to provide general spectral bounds for the unpreconditioned matrix $\sK$ and for cases where inexact Schur complements are used. 

In Section \ref{sec:unprec} we derive bounds on the eigenvalues of $\sK$. In Section \ref{sec:prec} we turn our attention to  preconditioned matrices of the form $\sM^{-1} \sK$.  In Section \ref{sec:prec_inexact} we allow the preconditioners to contain 
approximations of Schur complements and an approximate inversion of the leading block, and show how the bounds are affected as a result.
In Section \ref{sec:numex} we validate some of our analytical observations with numerical experiments. Finally, in Section \ref{sec:conclusions} we draw some conclusions.

\medskip

{\em Notation.}  The $n+m+p$ eigenvalues of the unpreconditioned and preconditioned double saddle-point matrices will be denoted by $\lambda$. We will use $\mu$ with appropriate matrix  superscripts and subscripts to denote eigenvalues of the matrix blocks, and $\sigma$ will accordingly signify singular values. The eigenvalues of $A$, for example, will be denoted by
$$ \mu_i^A, \quad i=1,\dots,n, $$ and in terms of ordering we will assume that 
$$ \mu_1 \geq \mu_2 \geq \cdots \geq \mu_n > 0.$$ To increase clarity, we will use $\mu_{\max}^A$ to denote $\mu_1^A$ and $\mu_{\min}^A$ to denote $\mu_n^A$, and so on. 

Based on the above conventions, we use the following notation for eigenvalues and singular values of the matrices comprising
$\sK$:
\smallskip
\begin{center}
	\begin{tabular}{|c|c|c|c|c|c|c|}
		\hline
		matrix  & size& type  & number & notation  & $\max$ & $\min$ \\
		\hline \hline
		$A$ & $n \times n$ & eigenvalues & $n$ & $\mu_i^A, \ i=1,\dots,n$ &   $\mu_{\max}^A$ & $\mu_{\min}^A$  \\ \hline
		$B$ & $m \times n$ & singular values & $m$ & $\sigma_i^B, \ i=1,\dots,m$ &   $\sigma_{\max}^B$ & $\sigma_{\min}^B$  \\ \hline
		$C$ & $p \times m$ & singular values & $p$ & $\sigma_i^C, \ i=1,\dots,p$ &   $\sigma_{\max}^C$ & $\sigma_{\min}^C$  \\ \hline
		$D$ & $m \times m$ & eigenvalues & $m$ & $\mu_i^D, \ i=1,\dots,m$ &   $\mu_{\max}^D$ & $\mu_{\min}^D$  \\ \hline
		$E$ & $p \times p$ & eigenvalues & $p$ & $\mu_i^D, \ i=1,\dots,p$ &   $\mu_{\max}^E$ & $\mu_{\min}^E$  \\ \hline
	\end{tabular}
\end{center}
\medskip

\section{Eigenvalue bounds for $\sK$}
\label{sec:unprec}

\subsection{Inertia and solvability conditions}
\label{sec:inertia}

We first discuss the inertia and conditions for nonsingularity of $\sK$ defined in \eqref{eqn:2spdef}. Recall that the inertia of a matrix is the triplet denoting the number of its positive, negative, and zero eigenvalues \cite[Definition 4.5.6]{hj85}.

\begin{proposition}
	The following conditions are necessary for $\sK$ to be invertible:
	\begin{enumerate}[label={(\roman*)}]
		\item $\ker(A) \cap \ker(B) = \emptyset$;
		\item $\ker(B^T) \cap \ker(D) \cap \ker(C) = \emptyset$;
		\item $\ker(C^T) \cap \ker(E) = \emptyset$.
	\end{enumerate}
	A sufficient condition for $\sK$ to be invertible is that $A$, $S_1$, and $S_2$ are invertible.
\end{proposition}

\begin{proof}
	We begin with the proof of statement (i) by assuming to the contrary that the intersection of the kernels is not empty -- namely, there exists a nonzero vector $x$ such that $Ax = Bx = 0$. This would mean that the block vector $\begin{bmatrix} x ^T & 0 & 0 \end{bmatrix}^T$ was a null vector of $\sK$, which would imply that $\sK$ was singular. Similar reasoning proves (ii) and (iii).
	
	For the sufficient condition, we observe that when $A$, $S_1$, and $S_2$ are invertible we can write a block-$LDL^T$ factorization of $\sK$:
	\begin{equation}
		\label{eqn:ldlt_k}
		\hspace{-2mm}
		\begin{bmatrix}
			A & B^T & 0 \\
			B & -D & C^T \\
			0 & C & E
		\end{bmatrix} =
		\begin{bmatrix}
			I & 0 & 0 \\
			BA^{-1} & I & 0 \\
			0 & -CS_1^{-1} & I
		\end{bmatrix}
		\underbrace{
			\begin{bmatrix}
				A & 0 & 0 \\
				0 & -S_1 & 0 \\
				0 & 0 & S_2
		\end{bmatrix}}_{=: \mathcal{D}}
		\begin{bmatrix}
			I & A^{-1}B^T & 0 \\
			0 & I & -S_1^{-1}C^T \\
			0 & 0 & I
		\end{bmatrix}.
	\end{equation}
	When $A$ and $S_1$ are invertible, $S_2$ is well-defined and $\sK$ is invertible if and only if $\mathcal{D}$ is invertible. The stated result follows.
\end{proof}

Throughout the rest of this paper, we assume that the sufficient condition holds: namely, that $A$, $S_1$, $S_2$ are invertible. Given our assumptions that $D$ and $E$ are semidefinite, this is equivalent to $A$, $S_1$, and $S_2$ being positive definite. This allows us to obtain the following result on the inertia of $\sK$, which will be useful in deriving our bounds.

\begin{lemma}[Inertia of a double saddle-point matrix]
	\label{lem:inertia1}
	The matrix $\sK$ has $n+p$ positive eigenvalues and $m$ negative eigenvalues.
\end{lemma}

\begin{proof}
	The matrices $A$, $S_1$, and $S_2$ are symmetric positive definite, by our assumptions. 
	Sylvester's Law of Inertia tells us the inertia of $\sK$ is the same as that of $\mathcal{D}$ defined in \eqref{eqn:ldlt_k}; the stated result follows.
\end{proof}

\begin{corollary}
	\label{cor:cubic-poly}
	Let $a, b, c, d,$ and $e$ be scalars with $a > 0$, $d, e \ge 0$, $s_1 := d + \frac{b^2}{a} > 0$ and $s_2 := e + \frac{c^2}{s_1} > 0$. Any cubic polynomial of the form
	$$
	p(\lambda) = \lambda^3 + (d-a-e)\lambda^2 + (ae -ad -de - b^2 - c^2)\lambda + (ade +ac^2 +b^2e)
	$$
	has two positive real roots and one negative real root. 
\end{corollary}

\begin{proof}
	Consider the $3 \times 3$ matrix
	$$
	P = \begin{bmatrix}
		a & b & 0 \\
		b & -d & c \\
		0 & c & e
	\end{bmatrix}.
	$$
	Using the  characteristic polynomial of $P$, it is straightforward to confirm that
	\begin{align*}
		\det(\lambda I-P) &= \lambda^3 - \Tr(P)\lambda^2 - \frac{1}{2}\left( \Tr(P^2) - \Tr^2(P) \right)\lambda - \det(P) \\
		&= p(\lambda).
	\end{align*}
	Because $P$ is symmetric its eigenvalues are real, and because $P$ is a double saddle-point matrix with $n = m = p =1$, the two positive and one negative root follow by Lemma \ref{lem:inertia1}. See Figure~\ref{fig:cubic} for a graphical illustration.
\end{proof}

\begin{figure}
	\centering
	\includegraphics[width=0.4\textwidth]{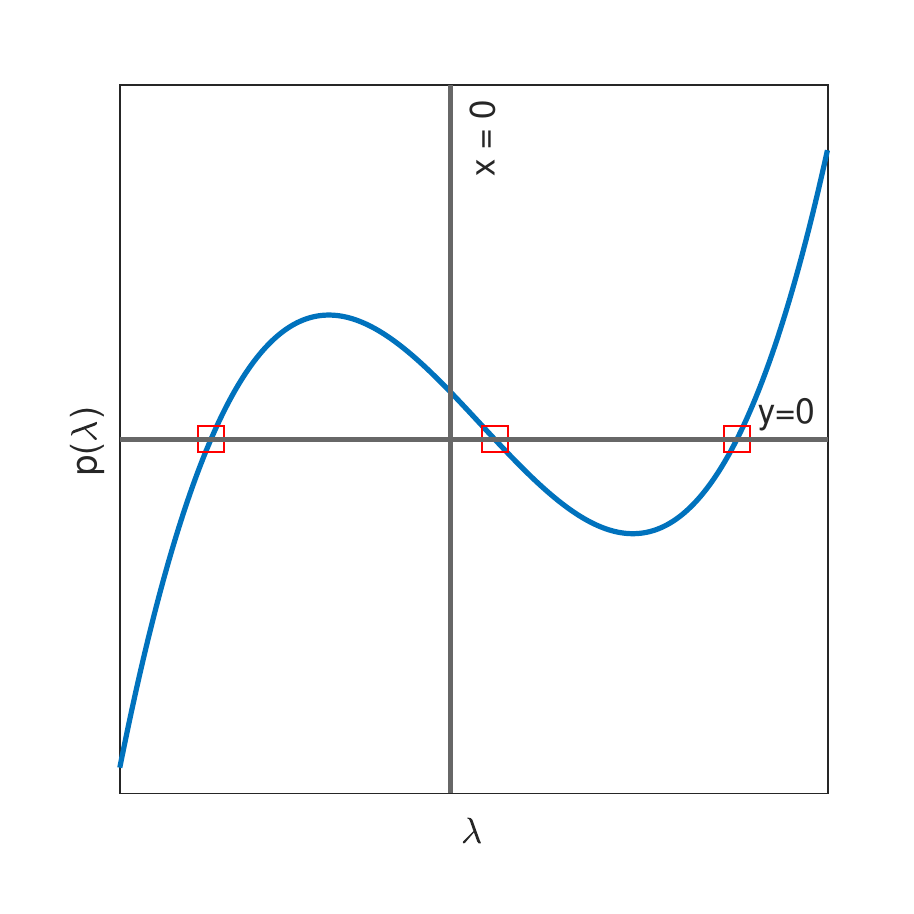}
	\caption{Plot of a cubic polynomial $p(\lambda)$ of the form described in Corollary \ref{cor:cubic-poly}, with two positive roots and one negative root. \label{fig:cubic}}
\end{figure}

\subsection{Derivation of bounds}
\label{sec:derivationK}	
Let us define three cubic polynomials, as follows:
\begin{subequations}
	\begin{align}
		p(\lambda) =  \lambda^3 + (\mu_{\max}^D - \mu_{\min}^A)\lambda^2 - \left(\mu_{\min}^A \mu_{\max}^D + (\sigma_{\max}^B)^2 + (\sigma_{\min}^C)^2 \right)\lambda + \mu_{\min}^A(\sigma_{\min}^C)^2;
		\label{eq:p}
	\end{align}
	\begin{align}
		\begin{split}
			q(\lambda)  =  & \lambda^3 + (\mu^D_{\min} - \mu^A_{\max} - \mu^E_{\max})\lambda^2 \\ & +  \left( \mu^A_{\max}\mu^E_{\max} - \mu^A_{\max}\mu^D_{\min} - \mu^D_{\min}\mu^E_{\max} - (\sigma^B_{\max})^2 - (\sigma^C_{\max})^2   \right) \lambda  \\
			& + \left( \mu^A_{\max}\mu^D_{\min}\mu^E_{\max} + \mu_{\max}^{A}(\sigma^C_{\max})^2+ (\sigma^B_{\max})^2\mu^E_{\max} \right);
		\end{split}
		\label{eq:q}
	\end{align}
	
	\begin{align}
		\begin{split}
			r(\lambda) = & \lambda^3 + (\mu^D_{\max} -\mu^A_{\min} -  \mu^E_{\min})\lambda^2 \\ & + \left( \mu^A_{\min}\mu^E_{\min} - \mu^A_{\min}\mu^D_{\max} - \mu^D_{\max}\mu^E_{\min} - (\sigma^B_{\max})^2 - (\sigma^C_{\max})^2    \right) \lambda \\
			& + (\mu^A_{\min}\mu^D_{\max}\mu^E_{\min} + \mu_{\min}^A(\sigma^C_{\max})^2  + (\sigma^B_{\max})^2 \mu_{\min}^E).
		\end{split}
		\label{eq:r}
	\end{align}
	\label{eq:pqr}
\end{subequations}

All three of these polynomials are of the form described in Corollary \ref{cor:cubic-poly}. Thus, all roots are real and each polynomial has two positive roots and one negative root. For notational convenience, we will denote the negative root of $p(\lambda)$, for example, by $p^-$, and use subscripts {\em max} and {\em min} to distinguish between the two positive roots. For example,  $p^+_{\max}$ will denote the largest positive root and $p^+_{\min}$ will denote the smallest positive root. The same notational rules apply to $q(\lambda)$ and $r(\lambda)$.

\begin{theorem}[Eigenvalue bounds, matrix $\sK$]
	\label{thm:bounds_unprec}
	Using the notation established  in \eqref{eq:pqr},
	the eigenvalues of $\sK$ are bounded within the intervals
	\begin{equation} 
		\left[r^-, \frac{\mu^A_{\max} - \sqrt{(\mu^A_{\max})^2 + 4(\sigma^B_{\min})^2}}{2} \right] \ \bigcup \ \left[ p^+_{\min}, q^+_{\max} \right]. 
		\label{eq:bound}
	\end{equation}
\end{theorem}

\begin{proof}
	{\fbox {\em Upper bound on positive eigenvalues.} }
	We let $v = \begin{bmatrix} x^T & y^T & z^T \end{bmatrix}^T$ be a vector with $x \in \mathbb{R}^n$, $y \in \mathbb{R}^m$, and $z \in \mathbb{R}^p$. Because $\sK$ is symmetric, we can derive an upper bound on the eigenvalues of $\sK$ by bounding the value of $\frac{v^T \sK v}{v^T v}$. We can write
	\begin{equation}
		\label{eqn:vtkv}
		v^T \sK v = x^T A x - y^T D y + z^T E z + 2 x^T B^T y + 2 y^T C^T z.
	\end{equation}
	We use Cauchy-Schwarz to bound the mixed bilinear forms -- such as $x^T B^T y$ -- by, for example,
	$$
	-||x|| \cdot ||B^T y|| \le x^T B^T y \le ||x|| \cdot ||B^T y||.
	$$
	Using this and the eigenvalues/singular values of the block of $\sK$, we can bound \eqref{eqn:vtkv} from above by
	\begin{align*}
		v^T \sK v &\le \mu^A_{\max}||x||^2 - \mu^{D}_{\min}||y||^2 + \mu^E_{\max}||z||^2 + 2 \sigma^B_{\max} ||x|| \cdot ||y|| + 2 \sigma^C_{\max} ||y|| \cdot ||z|| \\
		&= \begin{bmatrix}
			||x|| & ||y|| & ||z||
		\end{bmatrix}
		\underbrace{\begin{bmatrix}\mu^A_{\max} & \sigma^B_{\max} & 0 \\
				\sigma^B_{\max} & -\mu^D_{\min} & \sigma^C_{\max} \\
				0 & \sigma^C_{\max} & \mu^E_{\max}
		\end{bmatrix}}_{=: R} \begin{bmatrix}
			||x|| \\
			||y|| \\
			||z||
		\end{bmatrix}.
	\end{align*}
	An upper bound on $\frac{v^T \sK v}{v^T v}$ is therefore given by the maximal eigenvalue of $R$. The largest positive eigenvalue of $\sK$ is therefore less than or equal to the largest root of the characteristic polynomial $\det{( \lambda I - R)}$, which yields the desired result.
	
	{\fbox {\em Lower bound on negative eigenvalues.} }
	The proof is similar to that for the upper bound on the positive eigenvalues. Using Cauchy-Schwarz and the eigenvalues/singular values of the blocks, we bound \eqref{eqn:vtkv} from below by writing
	\begin{align*}
		v^T \sK v &\ge \mu^A_{\min}||x||^2 - 2 \sigma^B_{\max}||x|| \cdot ||y|| - \mu^D_{\max}||y||^2 - 2 \sigma^C_{\max}||y|| \cdot ||z|| + \mu^E_{\min} ||z||^2 \\
		&= \begin{bmatrix}
			||x|| & ||y|| & ||z||
		\end{bmatrix}
		\underbrace{\begin{bmatrix}\mu^A_{\min} & -\sigma^B_{\max} & 0 \\
				-\sigma^B_{\max} & -\mu^D_{\max} & -\sigma^C_{\max} \\
				0 & -\sigma^C_{\max} & \mu^E_{\min}
		\end{bmatrix}}_{=: R} \begin{bmatrix}
			||x|| \\
			||y|| \\
			||z||
		\end{bmatrix}.
	\end{align*}
	A lower bound on $\frac{v^T \sK v}{v^T v}$ is therefore given by the smallest eigenvalue of $R$. Taking the characteristic polynomial of $R$ yields the stated result.
	
	{\fbox {\em Upper bound on negative eigenvalues.} }
	We derive an upper bound on the negative eigenvalues of $\sK$ by finding a lower bound on the negative eigenvalues of $\sK^{-1}$. We begin by partitioning $\sK$ as
	\begin{equation}
		\label{eqn:partition}
		\sK = \left[
		\begin{array}{c c | c }
			A & B^T & 0 \\
			B & -D & C^T \\
			\hline
			0 & C & E
		\end{array}
		\right]
		= \begin{bmatrix}
			\sK_2 & \bar{C}^T \\
			\bar{C} & E
		\end{bmatrix},
	\end{equation}
	where $\sK_2 = \begin{bmatrix} A & B^T \\ B & -D \end{bmatrix}$ and $\bar{C} = \begin{bmatrix} 0 & C \end{bmatrix}$. By \cite[Equation (3.4)]{bgl05},
	$$
	\sK^{-1} = \begin{bmatrix}
		\sK_2^{-1} + \sK_2^{-1} \bar{C}^T S_2^{-1} \bar{C} \sK_2^{-1} & -\sK_2^{-1} \bar{C} S_2^{-1} \\
		-S_2^{-1} \bar{C} \sK_2^{-1} & S_2^{-1}
	\end{bmatrix},
	$$
	where $S_2 = E + \bar{C} \sK_2^{-1} \bar{C}^T = E + C S_1^{-1}C^T$. Notice that
	$$
	\sK^{-1} = \begin{bmatrix}
		\sK_2^{-1} & 0 \\
		0 & 0
	\end{bmatrix} + \begin{bmatrix}
		\sK_2^{-1} \bar{C} \\
		-I
	\end{bmatrix}
	S_2^{-1}
	\begin{bmatrix}
		\bar{C}^T \sK_2^{-1} & -I
	\end{bmatrix}.
	$$
	Because the second term is positive semidefinite, we conclude that the eigenvalues of $\sK^{-1}$ are greater than or equal to the eigenvalues of $\begin{bmatrix} \sK_2^{-1} & 0 \\ 0 & 0\end{bmatrix}$. Thus, a lower bound on the negative eigenvalues of $\sK_2^{-1}$ is also a lower bound on the negative eigenvalues of $\sK^{-1}$. This means that an upper bound on the negative eigenvalues of $\sK_2$ is also an upper bound on the negative eigenvalues of $\sK$. The desired result now follows from Silvester and Wathen \cite[Lemma 2.2]{silvester94}. 
	
	{\fbox {\em Lower bound on positive eigenvalues.} } We begin by noting that, because $E$ is positive semidefinite, the eigenvalues of $\sK$ are greater than or equal to those of
	$$
	\sK_{E=0} = \begin{bmatrix}
		A & B^T & 0 \\
		B & -D & C^T \\
		0 & C & 0
	\end{bmatrix}.
	$$
	Moreover, $\sK$ and $\sK_{E=0}$ have the same inertia, by Theorem \ref{thm:inertia}. Thus, the smallest positive eigenvalue of $\sK$ is greater than or equal to the smallest positive eigenvalue of $\sK_{E=0}$. We therefore obtain a lower bound on the positive eigenvalues of $\sK$ by using energy estimates with $\sK_{E=0}$. The eigenvalue problem associated with $\sK_{E=0}$ is
	\begin{equation}
		\label{eqn:e0_eig_problem}
		\begin{bmatrix}
			A & B^T & 0 \\
			B & -D & C^T \\
			0 & C & 0
		\end{bmatrix}
		\begin{bmatrix}
			x \\
			y \\
			z
		\end{bmatrix} = \lambda \begin{bmatrix}
			x \\
			y \\
			z
		\end{bmatrix}.
	\end{equation}
	We pre-multiply the first block row by $x^T$, the second by $y^T$, and the third by $z^T$ to obtain
	\begin{subequations}
		\begin{align}
			x^T A x + x^T B^T y &=\lambda x^T x ;  \label{eq:e0_1} \\
			y^T B x - y^T D y + y^T  C^T z &= \lambda y^T y  ; \label{eq:e0_2} \\
			z^T C y  &= \lambda z^T z.  \label{eq:e0_3}
		\end{align}
		\label{eq:en_e0}
	\end{subequations}
	
	The third block row of \eqref{eqn:e0_eig_problem} gives $z = \frac{1}{\lambda}Cy$. The first block row gives
	$$
	B^T y = (\lambda I - A)x.
	$$
	We now consider two cases based on the value of $\lambda$.
	\paragraph{Case I: $\lambda < \mu_{\min}^A$:} If $\lambda < \mu_{\min}^A$, then $(\lambda I - A)$ is negative definite, so we can write
	$$
	x = (\lambda I - A)^{-1} B^T y.
	$$
	Substituting these into \eqref{eq:e0_2} and rearranging gives
	\begin{subequations}
		\begin{align}
			\lambda y^T y &= y^T B(\lambda I - A)^{-1}B^T y - y^T D y + \frac{1}{\lambda}y^T C^T C y \label{eq:nrg_y}\\
			&\ge \frac{(\sigma_{\max}^B)^2}{\lambda - \mu_{\min}^A} y^T y - \mu_{\max}^D y^T y + \frac{(\sigma_{\min}^C)^2}{\lambda}y^T y.
		\end{align}
	\end{subequations}
	Dividing by $y^T y$, using the fact that $\lambda > 0$ and $\lambda - \mu_{\min}^A <0$, and rearranging yields
	$$
	\underbrace{ \lambda^3 + (\mu_{\max}^D - \mu_{\min}^A)\lambda^2 - \left(\mu_{\min}^A \mu_{\max}^D + (\sigma_{\max}^B)^2 + (\sigma_{\min}^C)^2 \right)\lambda + \mu_{\min}^A(\sigma_{\min}^C)^2}_{=p(\lambda)} \le 0.
	$$
	By applying Corollary \ref{cor:cubic-poly} with $a = \mu_{\min}^A$, $b = \sigma_{\max}^B$, $c = \sigma_{\min}^C$, and $d = \mu_{\max}^D$, we know that this polynomial has two positive roots. Moreover, $p(\lambda)$ is negative between these two roots:
	this follows from the fact that $p(0) > 0$ and $\lim_{\lambda \rightarrow \infty}p(\lambda) = \infty$. Therefore, we conclude that in this case $\lambda$ is greater than or equal to the smaller positive root of $p(\lambda)$, namely $p^+_{\min}$.
	
	\paragraph{Case II: $\lambda \ge \mu_{\min}^A$:} If $\lambda \ge \mu_{\min}^A$, then $(\lambda I - A)$ may be indefinite (and possibly singular). However, we can still obtain a lower bound for $\lambda$ by observing that $\mu_{\min}^A$ is greater than $p^+_{\min}$. As stated earlier, $p(\lambda)$ has two positive roots and the value of $p(\lambda)$ is negative between those roots. Moreover, these are the only positive values of $\lambda$ for which $p(\lambda) < 0$. We observe, after simplification, that
	$$
	p(\mu_{\min}^A) = -(\sigma_{\max}^B)^2 \mu_{\min}^A < 0.
	$$
	Therefore, the bound $\lambda \ge p^+_{\min}$ also holds in this case, which completes the proof.
\end{proof}

\paragraph{Remark 1.} From Theorem~\ref{thm:bounds_unprec} we see that when $B$ and $C$  are rank deficient, the internal bounds are zero. Under mild conditions on the ranks and kernels of $D$ and $E$, the statement of the theorem and its proof may be revised to obtain nonzero internal bounds in this case. However, doing so is rather technical, and since the case of full rank $B$ and $C$ is common, further details on this end case are omitted.
\vspace{0.4cm}

The matrix $\sK_0$ is a special case of $\sK$, and bounds on its eigenvalues can be obtained as a direct consequence of Theorem \ref{thm:bounds_unprec}.

\begin{corollary}[Eigenvalue bounds, matrix $\sK_0$]
	\label{cor:bounds_unprec2}
	Define the following three cubic polynomials as special cases of $p, q$ and $r$ defined in \eqref{eq:pqr} with $D=E=0$:
	\begin{subequations}
		\begin{align}
			\hat{p}(\lambda) &=  \lambda^3  - \mu_{\min}^A \lambda^2 - \left( (\sigma_{\max}^B)^2 + (\sigma_{\min}^C)^2  \right)\lambda + \mu_{\min}^A(\sigma_{\min}^C)^2; 
			\label{eq:pt}\\
			\hat{q}(\lambda)  &=   \lambda^3 - \mu^A_{\max}  \lambda^2  -  \left(  (\sigma^B_{\max})^2 + (\sigma^C_{\max})^2  \right) \lambda  
			+  \mu_{A}^{\max}(\sigma^C_{\max})^2 ;	
			\label{eq:qt} \\
			\hat{r}(\lambda) &=  \lambda^3 - \mu^A_{\min} \lambda^2   - \left((\sigma^B_{\max})^2 + (\sigma^C_{\max})^2  \right) \lambda 
			+   \mu^A_{\min}(\sigma^C_{\max})^2 .
			\label{eq:rt}
		\end{align}
		\label{eq:ptqtrt}
	\end{subequations}
	Using  \eqref{eq:ptqtrt},
	the eigenvalues of $\sK_0$ are bounded within the intervals
	\begin{equation} 
		\left[\hat{r}^-, \frac{\mu^A_{\max} - \sqrt{(\mu^A_{\max})^2 + 4(\sigma^B_{\min})^2}}{2} \right] \ \bigcup \ \left[ \hat{p}^+_{\min}, \hat{q}^+_{\max} \right]. 
		\label{eq:bound2}
	\end{equation}
\end{corollary}
The proof of Corollary \ref{cor:bounds_unprec2} is omitted; it is similar to the proof of Theorem \ref{thm:bounds_unprec}, with simplifications arising from setting $D=E=0.$

\subsection{Tightness of the bounds}
While the tightness of the bounds we have obtained depends on the problem at hand, all bounds presented in this section are attainable, as we will demonstrate with small examples.
For the extremal bounds (upper positive and lower negative), we note that both bounds hold for all double saddle-point matrices with $n = m = p = 1$, as the characteristic polynomial of the $3 \times 3$ matrix is the same as the polynomial given in the upper positive and lower negative bounds of Theorem \ref{thm:bounds_unprec}.

For the upper bound on negative eigenvalues, we consider the example (with $n=m=2$, and $p=1$):
$$
\sK = \left[\begin{array}{c c | c c | c }
	\mu_{\max}^A & 0 & \sigma_{\min}^B & 0 & 0 \\
	0 & \mu_{\max}^A & 0 & \sigma_{\min}^B & 0 \\
	\hline
	\sigma_{\min}^B & 0 & 0 & 0 & 0 \\
	0 & \sigma_{\min}^B & 0 & -\mu^D & \sigma^C \\
	\hline
	0 & 0 & 0 & \sigma^C & \mu^E
\end{array}.
\right]
$$
We can permute the rows and columns of $\sK$ to obtain a block diagonal matrix:
$$
\sK_D = \left[\begin{array}{c c | c c c }
	\mu_{\max}^A & \sigma_{\min}^B & 0 & 0 & 0 \\
	\sigma_{\min}^B & 0 & 0 & 0 & 0 \\
	\hline
	0 & 0 & \mu_{\max}^A & \sigma_{\min}^B & 0 \\
	0 & 0 & \sigma_{\min}^B & -\mu^D & \sigma^C \\
	0 & 0 & 0 & \sigma^C & \mu^E
\end{array}
\right].
$$
The upper left block of $\sK_D$ has as an eigenvalue $\frac{\mu^A_{\max} - \sqrt{(\mu^A_{\max})^2 + 4(\sigma^B_{\min})^2}}{2}$, which is the bound given by Theorem \ref{thm:bounds_unprec}.

Finally, for the lower positive bound, we consider the matrix (with $n=m=p=2$)
$$
\sK = \left[\begin{array}{ c c | c c | c c}
	\mu_{\min}^A & 0 & \sigma_{\max}^B & 0 & 0 & 0 \\
	0 & \mu_{\min}^A & 0 & \sigma_{\max}^B & 0 & 0 \\
	\hline
	\sigma_{\max}^B & 0 & -\mu_{\max}^D & 0 & \sigma_{\min}^C & 0 \\
	0 & \sigma_{\max}^B & 0 & -\mu_{\max}^D & 0 & \sigma_{\min}^C \\
	\hline
	0 & 0 & \sigma_{\min}^C & 0 & 0 & 0 \\
	0 & 0 & 0 & \sigma_{\min}^C & 0 & \mu^E
\end{array}\right].
$$
As in the previous example, we can permute the rows and columns to obtain a block diagonal matrix:
$$
\sK_D = \left[ \begin{array}{c c c | c c c}
	\mu_{\min}^A & \sigma_{\max}^B & 0 & 0 & 0 & 0\\
	\sigma_{\max}^B & -\mu_{\max}^D & \sigma_{\min}^C & 0 & 0 & 0 \\
	0 & \sigma_{\min}^C & 0 & 0 & 0 & 0 \\
	\hline
	0 & 0 & 0 & \mu_{\min}^A & \sigma_{\max}^B & 0 \\
	0 & 0 & 0 & \sigma_{\max}^B & -\mu_{\max}^D & \sigma_{\min}^C \\
	0 & 0 & 0 & 0 & \sigma_{\min}^C & \mu^E
\end{array}
\right].
$$
The characteristic polynomial of the upper left block is precisely $p(\lambda)$ defined in \eqref{eq:p}; thus, the bound of Theorem \ref{thm:bounds_unprec} is obtained.

\section{Eigenvalue bounds for $\sM^{-1} \sK$}
\label{sec:prec}

We now derive bounds for the preconditioned matrix $\sM^{-1} \sK$, with $\sM$ defined in \eqref{eq:M}.

\subsection{Inertia and eigenvalue multiplicity}
\label{sec:mk_inertia}
We begin with some observations on inertia and eigenvalue multiplicity. To simplify the presentation and proof of this result, we restrict ourselves to the case that $B$ and $C$ have full row rank. In some later results, we will lift this restriction to allow for rank-deficient $B$ and $C$.
\begin{theorem}[Inertia and algebraic multiplicity, matrix $\sM^{-1}\sK$]
	\label{thm:inertia} Let $\sK$ be defined as in \eqref{eqn:2spdef} and $\sM$ as in \eqref{eq:M}, and suppose that $B$ and $C$ have full row rank. The preconditioned matrix $\sM^{-1}\sK$ has:
	\begin{enumerate}[label=(\alph*)]
		\item[(i)] $m$ negative eigenvalues;
		\item[(ii)] $p$ eigenvalues in $(0,1)$;
		\item[(iii)] $n-m$ eigenvalues equal to $1$; and
		\item[(iv)] $m$ eigenvalues greater than $1$.
	\end{enumerate}
\end{theorem}

\begin{proof}
	Because $\sM$ is symmetric positive definite, $\sM^{1/2}$ exists and is invertible, and the inertias of $\sM^{-1/2}\sK \sM^{-1/2}$
	and $\sM^{-1}\sK$ are equal to the inertia of $\sK$.	 Thus, Lemma \ref{lem:inertia1} establishes that $\sM^{-1}\sK$ has $n+p$ positive and $m$ negative eigenvalues, which proves (i).
	
	For (ii)-(iv) we split the $n+p$ positive eigenvalues into eigenvalues less than, equal to, or greater than 1. For this we compute the inertia of the shifted, split-preconditioned matrix
	$$
	\sM^{-1/2}\sK \sM^{-1/2} - I = \begin{bmatrix}
		0 & \tilde{B}^T & 0 \\
		\tilde{B} & -\tilde{D}-I & \tilde{C}^T \\
		0 & \tilde{C} & \tilde{E} - I
	\end{bmatrix},
	$$
	where $\tilde{B} = S_1^{-1/2}BA^{-1/2}$, $\tilde{C}= S_2^{-1/2}CS_1^{-1/2}$, $\tilde{D} = S_1^{-1/2}DS_1^{-1/2}$, $\tilde{E} = S_2^{-1/2}ES_2^{-1/2}$. The number of positive, negative, and zero eigenvalues of $\sM^{-1/2}\sK \sM^{-1/2} - I$ will be equal to the number of eigenvalues of $\sM^{-1}\sK$ greater than, less than, or equal to 1, respectively.
	
	Noting that $\tilde{E} + \tilde{C}\tilde{C}^T = I$, we write 
	$$
	\sM^{-1/2}\sK \sM^{-1/2} - I = \begin{bmatrix}
		0 & \mathcal{B} \\
		\label{eqn:defT}
		\mathcal{B}^T & \mathcal{T}
	\end{bmatrix},
	$$
	where
	$$
	\mathcal{T} := \begin{bmatrix}
		-\tilde{D}-I & \tilde{C}^T \\
		\tilde{C} & -\tilde{C}\tilde{C}^T
	\end{bmatrix}
	\ \ \ \textrm{and} \ \ \
	\mathcal{B} = \begin{bmatrix}
		\tilde{B}^T & 0
	\end{bmatrix}.
	$$
	The matrix $\mathcal{B}$ is in $\mathbb{R}^{n \times (m+p)}$ and has rank $m$. We define a matrix
	$$
	N := \begin{bmatrix}
		\textbf{0}_{m \times p} \\
		I_p
	\end{bmatrix},
	$$
	where $\textbf{0}_{m \times p}$ is the $m \times p$ zero matrix and $I_p$ is the $p \times p$ identity matrix. The columns of $N$ form a basis for $\ker(\mathcal{B})$. Denote the inertia of $M$
	by $\In(M) = (n_{+}, n_{-}, n_{0})$ . It is well known (see, e.g., \cite[Lemma 3.4]{gould85}) that
	$$
	\In(\sM^{-1/2}\sK \sM^{-1/2} - I) = \In(N^T \mathcal{T} N) + (m, m, n-m).
	$$
	Because $N^T \mathcal{T}N = -\tilde{C}\tilde{C}^T$ is negative definite, this gives $\In(\sM^{-1/2}\sK \sM^{-1/2} - I) = (m, m+p, n-m)$, which yields (ii)-(iv).
\end{proof}

\subsection{Derivation of bounds}
\label{sec:prec_eig}

It is possible to obtain eigenvalue bounds on the preconditioned system $\sM^{-1}\sK$ by using the results of Theorem \ref{thm:bounds_unprec} on the (symmetric) preconditioned system $\sM^{-1/2}\sK\sM^{-1/2}$; however, some of the resulting bounds will be loose. The reason for this is that Theorem \ref{thm:bounds_unprec} assumes no relationships between the blocks of $\sK$ and considers each block individually. In this section, we will derive tight eigenvalue bounds using energy estimates, which allow us to fully exploit the relationships between the blocks of the preconditioned system.

We begin by recalling the following result from Horn and Johnson \cite[Theorem 7.7.3]{hj85}, which will be useful throughout the analysis that follows.
\begin{lemma}
	Let $M$ and $N$ be symmetric positive semidefinite matrices such that $M+N$ is positive definite. Then $(M+N)^{-1}M$ is symmetric positive semidefinite with all eigenvalues in $[0,1]$.
	\label{lem:inv_times_sum}
\end{lemma}

\subsection*{The case $D=E=0$}
\label{sec:prec_eig_unreg}
When $D$ and $E$ are both zero, $\sM^{-1} \sK$ has six distinct eigenvalues given by: $1$, $\frac{1 \pm \sqrt{5}}{2}$, and the roots of the cubic polynomial $\lambda^3 -\lambda^2 -2\lambda + 1$. We refer to Cai et al. \cite[Theorem 2.2]{cai21} for proof.

\subsection*{The case $D = 0$ and $E \ne 0$}
When $D = 0$, it is necessary that $B$ have full row rank in order for $S_1$ to be invertible; however, $C$ may be rank-deficient. Suppose that $C^T$ has nullity $k$. The following result holds.

\begin{theorem}[Eigenvalue bounds, matrix $\sM^{-1}\sK$, $D = 0$, $E \ne 0$]
	\label{thm:bounds_prec_d0}
	When $D = 0$ and $E \ne 0$, the eigenvalues of $\mathcal{M}^{-1}\mathcal{K}$ are given by: $\lambda = 1$ with multiplicity $n-m+k$; $\lambda = \frac{1 \pm \sqrt{5}}{2}$, each with multiplicity $m-p+k$; and $p-k$ eigenvalues located in each of the three intervals: $\left[ \zeta^-, \frac{1 - \sqrt{5}}{2} \right)$, $\left[ \zeta_{\min}^+, 1 \right)$, and $\left( \frac{1 + \sqrt{5}}{2}, \zeta_{\max}^+ \right]$, where $\zeta^-$, $\zeta_{\min}^+$ and $\zeta_{\max}^+$ are the roots of the cubic polynomial $\lambda^3 - \lambda^2 - 2\lambda + 1$, ordered from smallest to largest. (These are approximately $-1.2470$, $0.4450$, and $1.8019$, respectively.)
\end{theorem}

\begin{proof}
	We write out the (left-)preconditioned operator
	$$
	\sM^{-1}\sK = \begin{bmatrix}
		I & A^{-1}B^T & 0 \\
		S_1^{-1}B & 0 & S_1^{-1}C^T \\
		0 & S_2^{-1}C & S_2^{-1}E
	\end{bmatrix},
	$$
	with corresponding eigenvalue equations
	\begin{subequations}
		\begin{align}
			\label{eqn:eig_kd0_1}
			x + A^{-1}B^Ty &= \lambda x; \\
			\label{eqn:eig_kd0_2}
			S_1^{-1}Bx + S_1^{-1}C^T z &= \lambda y; \\
			\label{eqn:eig_kd0_3}
			S_2^{-1}C y + S_2^{-1}E z &= \lambda z.
		\end{align}
	\end{subequations}
	
	We obtain $n-m$ eigenvectors for $\lambda = 1$ by choosing $x \in \ker(B)$ and $y,z=0$. By considering $y \in \ker(C)$ and $z = 0$, we obtain eigenvalues $\lambda = \frac{1 \pm \sqrt{5}}{2}$, each with geometric multiplicity $m - p + k$.
	
	For the remaining eigenvalues, we assume that $z \ne 0$ and $\lambda \notin \{ 1, \frac{1 \pm \sqrt{5}}{2}\}$. From \eqref{eqn:eig_kd0_1} we obtain
	$$
	x = \frac{1}{\lambda-1} A^{-1}B^T y,
	$$
	which we substitute into \eqref{eqn:eig_kd0_2} and rearrange to get
	$$
	y = \frac{\lambda - 1}{\lambda^2 - \lambda - 1}S_1^{-1}C^T z.
	$$
	Substituting this into \eqref{eqn:eig_kd0_3} gives
	\begin{equation}
		\label{eqn:s2z}
		\frac{\lambda - 1}{\lambda^2 - \lambda - 1}S_2^{-1}CS_1^{-1}C^T z + S_2^{-1} E z = \lambda z.
	\end{equation}
	Because $S_2 = E + CS_1^{-1}C^T$, we can write $S_2^{-1} E = I - S_2^{-1}CS_1^{-1}C^T$. We substitute this into \eqref{eqn:s2z} and rearrange to obtain
	$$
	\bigg(-\lambda^2 + 2 \lambda \bigg)S_2^{-1}CS_1^{-1}C^T z = \bigg(\lambda^3 - 2\lambda^2 + 1 \bigg)z.
	$$
	Let $\{\delta_j, v_j\}$, for $1 \le j \le p$, denote an eigenpair of $S_2^{-1}CS_1^{-1}C^T$. For $k$ of these eigenpairs corresponding to $v_j \in \ker(C^T)$, we note that
	$$
	S_2^{-1}Ev_j = (I-S_2^{-1}CS_1^{-1}C^T)v_j = v_j,
	$$
	and therefore $\begin{bmatrix} 0 & 0 & v_j^T \end{bmatrix}^T$ is an eigenvector of $\sM^{-1}\sK$ with $\lambda = 1$.
	For the $p-k$ remaining eigenpairs, we have $0 < \delta_j \le 1$ (by Lemma \ref{lem:inv_times_sum}).  We can then write
	$$
	\bigg(-\lambda_j^2 + 2 \lambda_j \bigg)\delta_j z_j = \bigg(\lambda_j^3 - 2\lambda_j^2 + 1 \bigg)z_j,
	$$
	where $\lambda_j$ corresponds to an eigenpair $\{ \delta_j, z_j\}$ of $S_2^{-1}CS_1^{-1}C^T$. Because $z_j \ne 0$ this implies that
	\begin{equation}
		\label{eqn:cubic_j}
		\lambda_j^3 - (2-\delta_j)\lambda_j^2 - 2\delta_j \lambda_j + 1 = 0.
	\end{equation}
	Thus, each of the $p-k$ positive eigenvalues $\delta_j$ of $S_2^{-1}CS_1^{-1}C^T$ yields three distinct corresponding eigenvalues $\lambda_j^{(1)}, \lambda_j^{(2)}$, and $\lambda_j^{(3)}$ of $\mathcal{M}^{-1}\mathcal{K}$, corresponding to the roots of the cubic polynomial \eqref{eqn:cubic_j}. These $3(p-k)$ eigenvalues, combined with the eigenvalues described earlier, account for all eigenvalues of $\mathcal{M}^{-1}\mathcal{K}$. Substituting $\delta_j = 0$ and $\delta_j = 1$ into \eqref{eqn:cubic_j} gives us the three intervals for these eigenvalues stated in the theorem.
\end{proof}

\subsection*{The case $D \ne 0$ and $E = 0$}
When $D \ne 0$ and $E = 0$, the bounds are the same as when $D$ and $E$ are both nonzero, which are given next.

\subsection*{The case $D, E \ne 0$}
\label{sec:prec_eig_reg}

\begin{theorem} [Eigenvalue bounds,  matrix $\sM^{-1}\sK$, $D, E \ne 0$]
	\label{thm:bounds_prec}
	The eigenvalues of $\sM^{-1}\sK$ are bounded within the intervals 
	$$
	\left[ -\frac{1+\sqrt5}{2}, \frac{1-\sqrt{5}}{2}\right] \cup \left[ \zeta_{\min}^+, \zeta_{\max}^+ \right],
	$$
	where $\zeta_{\min}^+$ and $\zeta_{\max}^+$ are respectively the smallest and largest positive roots of the cubic polynomial $\lambda^3 - \lambda^2 - 2\lambda + 1$ (approximately 0.4450 and 1.8019, respectively). 
\end{theorem}

\begin{proof}
	{\fbox{\em Upper bound on positive eigenvalues.}}
	We know from Theorem \ref{thm:inertia} that the upper bound on positive eigenvalues is greater than 1, so we assume here that $\lambda > 1$. We begin by writing the eigenvalue equations as
	\begin{subequations}
		\begin{align}
			\label{eqn:eigeq1}
			Ax + B^T y &= \lambda A x; \\
			\label{eqn:eigeq2}
			Bx - Dy + C^T z &= \lambda S_1 y; \\
			\label{eqn:eigeq3}
			Cy + Ez &= \lambda S_2 z.
		\end{align}
	\end{subequations}
	From \eqref{eqn:eigeq1} we get 
	\begin{equation}
		\label{eqn:solve_for_x}
		x = \frac{1}{\lambda - 1}A^{-1}B^T y,
	\end{equation}
	and from \eqref{eqn:eigeq3} we get $z = (\lambda S_2 - E)^{-1} Cy$ (because $\lambda > 1$, we are guaranteed that $\lambda S_2 - E = (\lambda -1)E +\lambda CS_1^{-1}C^T $ is positive definite). Substituting these values back into \eqref{eqn:eigeq2} and pre-multiplying by $y^T$ gives
	\begin{equation*}
		\left(\frac{1}{\lambda-1}\right)y^T BA^{-1}B^T y + y^T (D+\lambda S_1) y- y^T C^T (\lambda S_2 - E)^{-1}Cy = 0.
	\end{equation*}
	Recalling that $S_1 = D + BA^{-1}B^T$ we can rewrite this as
	\begin{equation}
		\label{eqn:energy_y}
		\left( \frac{1}{1-\lambda} + \lambda \right) y^T S_1 y + \left( 1 - \frac{1}{1-\lambda} \right) y^T D y - y^T C^T (\lambda S_2 - E)^{-1} C y = 0.
	\end{equation}
	Because $\lambda > 1$, we have
	\begin{align*}
		y^T C^T (\lambda S_2 - E)^{-1} C y &\le y^T C^T \Big(\lambda (S_2 - E)\Big)^{-1} C y\\
		&= \frac{1}{\lambda}y^T C^T \left(CS_1^{-1}C^T\right)^{-1}C y.
	\end{align*}
	Therefore, \eqref{eqn:energy_y} gives
	\begin{equation}
		\label{eqn:pre-yt}
		\left( \frac{1}{1-\lambda} + \lambda \right) y^T S_1 y + \left( 1 - \frac{1}{1-\lambda} \right) y^T D y- \frac{1}{\lambda} y^T C^T (CS_1^{-1} C^T)^{-1}Cy \le 0.
	\end{equation}
	Next, we let $\tilde{y} = S_1^{1/2}y$ and rewrite the first and third terms in \eqref{eqn:pre-yt} in terms of $\tilde{y}$:
	\begin{equation}
		\label{eqn:energy_yt}
		\left( \frac{1}{1-\lambda} + \lambda \right) \tilde{y}^T \tilde{y} + \left( 1 - \frac{1}{1-\lambda} \right) y^T D y- \frac{1}{\lambda} \tilde{y}^T \underbrace{S_1^{-1/2} C^T (CS_1^{-1} C^T)^{-1}C S_1^{-1/2}}_{:= P}\tilde{y} \le 0.
	\end{equation}
	Because $P$ defined in \eqref{eqn:energy_yt} is an orthogonal projector, we have $\tilde{y}^T P \tilde{y} \le \tilde{y}^T \tilde{y}$. And because $\lambda > 1$, we have $1 - \frac{1}{1-\lambda} >0$, which means that the second term $\left( 1 - \frac{1}{1-\lambda} \right) y^T D y$ is non-negative and can be dropped from the inequality. The inequality \eqref{eqn:energy_yt} therefore becomes
	$$
	\left( \frac{1}{1-\lambda} + \lambda \right) \tilde{y}^T \tilde{y} - \frac{1}{\lambda}\tilde{y}^T{y} \le 0,
	$$
	which we can divide by $\tilde{y}^T\tilde{y}$ rearrange to give
	$$
	\lambda^3 - \lambda^2 - 2\lambda +1 \le 0.
	$$
	This, combined with the assumption that $\lambda > 1$, gives us the stated result that $\lambda$ is less than or equal to the largest root of $\lambda^3 - \lambda^2 - 2\lambda +1$. The polynomial $\lambda^3 - \lambda^2 - 2\lambda +1$ is of the form given in Corollary \ref{cor:cubic-poly} with $a = b = c =1$, $d = e = 0$, so its value is negative between the two positive roots.
	
	{\fbox {\em Lower bound on negative eigenvalues.}}
	We begin from \eqref{eqn:energy_y} and note that when $\lambda < 0$, by similar reasoning as was shown for the upper bound on positive eigenvalues,
	$$
	\left( \frac{1}{1-\lambda} + \lambda \right) y^T S_1 y + \left( 1 - \frac{1}{1-\lambda} \right) y^T D y- \frac{1}{\lambda} y^T C^T (CS_1^{-1} C^T)^{-1}Cy \ge 0.
	$$
	Rewriting the inequality in terms of $\tilde{y} = S_1^{1/2}y$ gives
	$$
	\left( \frac{1}{1-\lambda} + \lambda \right) \tilde{y}^T \tilde{y} + \left( 1 - \frac{1}{1-\lambda} \right) \tilde{y}^T S_1^{-1/2} D S_1^{-1/2} \tilde{y}- \frac{1}{\lambda} \tilde{y}^T P \tilde{y} \ge 0,
	$$
	where $P$ is the orthogonal projector defined in \eqref{eqn:energy_yt}. For the second term, note that $S_1^{-1/2} D S_1^{-1/2}$ is similar to $S_1^{-1}D$ which has all eigenvalues between 0 and 1 (by Lemma \ref{lem:inv_times_sum}). Thus, both $\tilde{y}^T S_1^{-1/2} D S_1^{-1/2} \tilde{y}$ and $\tilde{y}^T P \tilde{y}$ are less than or equal to $\tilde{y}^T \tilde{y}$ and we can therefore write
	$$
	\left( \frac{1}{1-\lambda} + \lambda \right) \tilde{y}^T \tilde{y} + \left( 1 - \frac{1}{1-\lambda} \right)\tilde{y}^T\tilde{y}- \frac{1}{\lambda}\tilde{y}^T \tilde{y} \ge 0,
	$$
	which, after dividing by $\tilde{y}^T \tilde{y}$ and simplifying, yields
	$$
	\lambda^2 + \lambda - 1 \le 0.
	$$
	This along with the assumption that $\lambda < 0$ gives the desired bound of $\lambda \ge -\frac{1 + \sqrt{5}}{2}$.
	
	{\fbox {\em Lower bound on positive eigenvalues.}}
	Assume that $0 < \lambda < 1$ (we know from Theorem \ref{thm:inertia} that the lower bound is in this interval). Substituting \eqref{eqn:solve_for_x} into \eqref{eqn:eigeq2} and solving for $y$ gives
	\begin{equation}
		\label{eqn:moreQ}
		y=\underbrace{\left( \frac{1}{1-\lambda} BA^{-1}B^T + D + \lambda S_1 \right)^{-1}}_{=: Q} C^T z,
	\end{equation}
	When $0 < \lambda < 1$, the value $\frac{1}{1-\lambda}$ is positive, so we are guaranteed that $Q$ in \eqref{eqn:moreQ} is positive definite. If $z \in \ker(C^T)$, \eqref{eqn:moreQ} gives $y=0$ and \eqref{eqn:solve_for_x} gives $x=0$, and we can see from the eigenvalue equations \eqref{eqn:eigeq1}-\eqref{eqn:eigeq3} that this eigenvector corresponds to $\lambda = 1$ (because $Ez = S_2z$ for $z \in \ker(C^T)$). This contradicts our assumption that $\lambda < 1$; thus, we assume $z \notin \ker(C^T)$.  
	
	We can then write \eqref{eqn:eigeq3} as
	\begin{equation}
		\label{eqn:withD}
		C\left( \frac{1}{1-\lambda}BA^{-1}B^T + D + \lambda S_1 \right)^{-1}C^T z + (1 - \lambda)E z - \lambda CS_1^{-1}C^T z = 0.
	\end{equation}
	When $0 < \lambda < 1$, we have $\frac{1}{1-\lambda} > 1$. Therefore, if we take the inner product of $z^T$ with \eqref{eqn:withD}, replace $D$ by $\frac{1}{1-\lambda}D$, and drop the non-negative term $(1 - \lambda)z^T E z$, we obtain the inequality
	$$
	z^T C \left( \left( \frac{1}{1-\lambda} + \lambda \right) S_1 \right)^{-1}C^T z - \lambda z^T CS_1^{-1}C^T z \le 0.
	$$
	After simplifying and dividing by $z^T CS_1^{-1}C^T z$, we obtain
	\begin{equation}
		\label{eqn:cubic}
		\lambda^3 - \lambda^2 -2\lambda + 1 \le 0.
	\end{equation}
	This, combined with the assumption that $0 < \lambda < 1$, gives us that $\lambda$ must be greater than or equal to the smaller positive root of \eqref{eqn:cubic}, as required.

	{\fbox {\em Upper bound on negative eigenvalues.}}
	Assume that $\lambda < 0$. We begin from Equation \eqref{eqn:energy_y} and note that $\lambda S_2 -E$ is negative definite. Therefore,
	$$
	\left( \frac{1}{1-\lambda} + \lambda \right)y^T S_1 y + \left( 1 - \frac{1}{1-\lambda} \right) y^T D y \le 0.
	$$
	Since $1 - \frac{1}{1 - \lambda} > 0$, the $y^T D y$ term is non-negative and can be dropped while keeping the inequality, giving us
	$$
	\left( \frac{1}{1-\lambda} + \lambda \right) y^T S_1 y \le 0.
	$$
	We thus require
	$$
	\frac{1}{1-\lambda} + \lambda \le 0 \ \ \textrm{with } \ \lambda < 0,
	$$
	which leads to the desired bound $\lambda \le \frac{1 - \sqrt{5}}{2}$.
\end{proof}

\section{Bounds for preconditioners with approximations of Schur complements}
\label{sec:prec_inexact}
In practice, it is too expensive to invert $A$, $S_1$, and $S_2$ exactly. In this section we examine eigenvalue bounds on matrices of the form $\tilde{\sM}^{-1}\sK$, where $\tilde{\sM}$ uses symmetric positive definite and ideally spectrally equivalent approximations for $A$, $S_1$, and $S_2$. Specifically, we consider the approximate block diagonal preconditioner
\begin{equation}
	\label{eqn:pdtilde}
	\tilde{\sM} = \begin{bmatrix}
		\tilde{A} & 0 & 0 \\
		0 & \tilde{S}_1 & 0 \\
		0 & 0 & \tilde{S}_2
	\end{bmatrix},
\end{equation}
with $\tilde{A}$, $\tilde{S}_1$ and $\tilde{S}_2$ satisfying the following:

\begin{assumption}
	\label{thm:spec_equiv}
	Let $\Lambda(\cdot)$ denote the spectrum of a matrix. The diagonal blocks of the approximate preconditioner $\tilde{\sM}$, given by $\tilde{A}$, $\tilde{S}_1$, and $\tilde{S}_2$, satisfy:
	\begin{enumerate}[label=\roman*.]
		\item $\Lambda(\tilde{A}^{-1}A) \in [\alpha_0, \beta_0]$;
		\item $\Lambda(\tilde{S}_1^{-1}S_1) \in [\alpha_1, \beta_1]$;
		\item $\Lambda(\tilde{S}_2^{-1}S_2) \in [\alpha_2, \beta_2]$,
	\end{enumerate}
	where $0 < \alpha_i \le 1 \le \beta_i$.
\end{assumption}
We note that to obtain spectral equivalence we seek approximations that yield values of $\alpha_i$ independent of the mesh size and bounded uniformly away from zero. It is also worth noting that the values of $\alpha_i$ and $\beta_i$ are typically not explicitly available. We briefly address this at the end of this section, following our derivation of the bounds.

To simplify our analyses, we define:
\begin{subequations}
	\begin{align}
		\label{eqn:q0}
		\tilde{A}^{-1}A =: Q_0; \\
		\label{eqn:q1}
		\tilde{S}_1^{-1}S_1 =: Q_1; \\
		\label{eqn:q2}
		\tilde{S}_2^{-1}S_2 =: Q_2.
	\end{align}
\end{subequations}

To derive the eigenvalue bounds, we consider the split preconditioned matrix
\begin{equation}
	\label{eqn:ptildek}
	\tilde{\sM}^{-1/2}\sK \tilde{\sM}^{-1/2} = \begin{bmatrix}
		\tilde{Q}_0 & \tilde{B}^T & 0 \\
		\tilde{B} & -\tilde{D} & \tilde{C}^T \\
		0 & \tilde{C} & \tilde{E}
	\end{bmatrix},
\end{equation}
where $\tilde{Q}_0 = \tilde{A}^{-1/2}A\tilde{A}^{-1/2}$, $\tilde{B} = \tilde{S}_1^{-1/2}B\tilde{A}^{-1/2}$, $\tilde{C} = \tilde{S}_2^{-1/2}C\tilde{S}_1^{-1/2}$, $\tilde{D} = \tilde{S}_1^{-1/2}D\tilde{S}_1^{-1/2}$, and $\tilde{E} = \tilde{S}_2^{-1/2}E\tilde{S}_2^{-1/2}$. We proceed by bounding the eigenvalues and singular values of the blocks of $\tilde{\sM}^{-1/2}\sK \tilde{\sM}^{-1/2}$ and then applying the results for general double saddle-point matrices presented in Section \ref{sec:derivationK}. In order to avoid providing internal eigenvalue bounds equal to zero (see Remark 1) we assume here that $B$ and $C$ have full row rank.

\begin{lemma}
	\label{lem:blockBounds}
	When $B$ and $C$ have full row rank, bounds on the eigenvalues/singular values of the blocks of the matrix $\tilde{\sM}^{-1/2}\sK\tilde{\sM}^{-1/2}$ are as follows:
	\begin{itemize}
		\item $\tilde{Q}_0$: eigenvalues are in $[\alpha_0, \beta_0]$.
		\item $\tilde{B}$: singular values are in $\left[ \sqrt{\frac{\alpha_0 \alpha_1}{1 + \eta_D}}, \sqrt{\beta_0 \beta_1} \right]$, where $\eta_D$ is the maximal eigenvalue of $(BA^{-1}B^T)^{-1}D$.
		\item $\tilde{C}$: singular values are in $\left[ \sqrt{\frac{\alpha_1 \alpha_2}{1 + \eta_E}}, \sqrt{\beta_1 \beta_2} \right]$, where $\eta_E$ is the maximal eigenvalue of $(CS_1^{-1}C^T)^{-1}E$.
		\item $\tilde{D}$: eigenvalues are in $[0, \beta_1]$.
		\item $\tilde{E}$: eigenvalues are in $[0, \beta_2]$.
	\end{itemize}
\end{lemma}

\begin{proof}
	The eigenvalue bounds on $\tilde{Q}_0$ follow from the fact that $\tilde{Q}_0$ is similar to $Q_0$.  For $\tilde{D}$ and $\tilde{E}$, the lower bounds follow from the fact that $D$ and $E$ are semidefinite (if $D$ and/or $E$ are definite, this bound will be loose, but we use the zero bound to simplify some results that we present in this section). For the upper bound on $\tilde{D}$, we note that $\tilde{D} = \tilde{S}_1^{-1/2}D \tilde{S}_1^{-1/2}$, which is similar to $\tilde{S}_1^{-1}D = Q_1S_1^{-1}D$. The eigenvalues of $S_1^{-1}D$ are less than or equal to 1 by Lemma \ref{lem:inv_times_sum}, 
	meaning that those of $Q_1 S_1^{-1}D$ are less than or equal to $\beta_1$. Analogous reasoning gives the upper bound for $\tilde{E}$.
	
	We now present the results for $\tilde{B}$. For the upper bound, note that the matrix $\tilde{B}\tilde{B}^T = \tilde{S}_1^{-1/2} B \tilde{A}^{-1} B^T \tilde{S}_1^{-1/2}$  is similar to $\tilde{S}_1^{-1}B \tilde{A}^{-1}B^T = Q_1 S_1^{-1}B\tilde{A}^{-1}B^T$. Because the eigenvalues of $Q_1$ are in $[\alpha_1, \beta_1]$, we need only bound the eigenvalues of $S_1^{-1}B\tilde{A}^{-1}B^T$. These are the same as the nonzero eigenvalues of
	$$
	\tilde{A}^{-1}B^T S_1^{-1}B = Q_0 A^{-1}B^T S_1^{-1}B.
	$$
	The nonzero eigenvalues of $A^{-1}B^T S_1^{-1}B$ are the same as those of $S_1^{-1}BA^{-1}B^T$, which are all less than or equal to 1 by Lemma \ref{lem:inv_times_sum}. 
	Thus, the eigenvalues of $Q_0 A^{-1}B^T S_1^{-1}B$ are less than or equal to $\beta_0$, from which we conclude that the eigenvalues of $\tilde{B}\tilde{B}^T$ are less than or equal to $\beta_0 \beta_1$, giving an upper singular value bound of $\sqrt{\beta_0 \beta_1}$. Similarly, a lower bound on the eigenvalues of $\tilde{B}\tilde{B}^T$ is given by $\alpha_0 \alpha_1$ times a lower bound on the eigenvalues of $S_1^{-1}BA^{-1}B^T$. Because we have assumed that $B$ is full rank, $BA^{-1}B^T$ is invertible, implying that
	$$
	S_1^{-1}BA^{-1}B^T = \left( (BA^{-1}B^T)^{-1}S_1 \right)^{-1} = \left( I+ (BA^{-1}B^T)^{-1}D \right)^{-1}.
	$$
	The stated result then follows because the eigenvalues of $S_1^{-1}BA^{-1}B^T$ are greater than or equal to $\frac{1}{1 + \eta_D}$, where $\eta_D$ is the maximal eigenvalue of $(BA^{-1}B^T)^{-1}D$. Thus, a lower bound on the singular values of $\tilde{B}$ is given by the square root of this value.
	
	The bounds for $\tilde{C}$ are obtained in the same way as those for $\tilde{B}$.
\end{proof}

We can now present bounds on the eigenvalues of $\tilde{\sM}^{-1/2}\sK \tilde{\sM}^{-1/2}$. We define three cubic polynomials:
\begin{subequations}
	\begin{align}
		u(\lambda) &=  \lambda^3 + (\beta_1 - \alpha_0) \lambda^2 - \left( \alpha_0 \beta_1  + \frac{\alpha_1 \alpha_2}{1+\eta_E} + \beta_0 \beta_1 \right)\lambda + \frac{\alpha_0 \alpha_1 \alpha_2}{1+\eta_E};
		\\
		v(\lambda) &= \lambda^3 - (\beta_0 + \beta_2) \lambda^2 + ( \beta_0 \beta_2- \beta_0\beta_1- \beta_1 \beta_2) \lambda + 2\beta_0 \beta_1 \beta_2;
		\\
		w(\lambda) &= \lambda^3 + (\beta_1 - \alpha_0)\lambda^2 - ( \alpha_0 \beta_1+\beta_0 \beta_1 + \beta_1 \beta_2) \lambda + \alpha_0 \beta_1 \beta_2.
	\end{align}
\end{subequations}
These polynomials all have two positive roots and one negative root, by Corollary \ref{cor:cubic-poly}. As in Section Section \ref{sec:unprec}, we let $u^-$ denote the (single) negative root of a polynomial $u$, and let $u_{\min}^+$ and $u_{\max}^+$ respectively denote the smallest and largest positive roots.
\begin{theorem}[Eigenvalue bounds, matrix $\tilde{\sM}^{-1/2}\sK \tilde{\sM}^{-1/2}$]
	\label{thm:inexact_precon_bounds}
	When $B$ and $C$ have full row rank, the eigenvalues of $\tilde{\sM}^{-1/2}\sK \tilde{\sM}^{-1/2}$ are bounded within the intervals
	\begin{equation} 
		\left[w^-, \frac{\beta_0 - \sqrt{\beta_0^2 + \frac{4\alpha_0 \alpha_1}{1 + \eta_D}}}{2} \right] \ \bigcup \ \left[ u^+_{\min}, v^+_{\max} \right]. 
		\label{eq:bound_approx}
	\end{equation}
\end{theorem}

\begin{proof}
	The stated bounds follow from Lemma \ref{lem:blockBounds} and Theorem \ref{thm:bounds_unprec}.
\end{proof}

\paragraph{Remark 2.}Obtaining $\eta_D$ and $\eta_E$ requires computation of eigenvalues of two matrices related to the Schur complements $S_1$ and $S_2$: $(S_1-D)^{-1} D=(B A^{-1} B^T)^{-1} D$ and $(S_2-E)^{-1} E=(C S_1^{-1} C^T)^{-1} E$, respectively. As we have previously mentioned, in practice when solving \eqref{eq:sp_system} the matrices $S_1$ and $S_2$ would typically be approximated by sparse and easier-to-invert matrices, rather than formed and computed explicitly. Therefore, we cannot expect to compute $\eta_D$ and $\eta_E$ exactly. Spectral equivalence relations may be helpful in providing reasonable approximations here. For example, in common formulations of the Stokes-Darcy problem \cite{cai09} the matrix $A$ is a discrete Laplacian and $B$ and $C$ are discrete divergence operators (or scaled variations thereof). In such a case, for certain finite element discretizations $B A^{-1} B^T=S_1-D$ is spectrally equivalent to the mass matrix \cite[Section 5.5]{elman05}. Denote it by $Q$. Then, estimating the maximal eigenvalue of $\eta_D$ amounts to computing an approximation to the maximal eigenvalue of $Q^{-1} D$, which is computationally  straightforward given the favorable spectral properties of  $Q$. If the above-mentioned Schur complement  is approximated by the mass matrix, then it can be shown that $C S_1^{-1} C^T=S_2-E$ is strongly related to the scalar Laplacian, and therefore maximal eigenvalue of $\eta_E$ would be relatively easy to approximate as well. 
\vspace{0.4cm}

When $D = 0$, $\eta_D$ and the maximal eigenvalue of $\tilde{D}$ are zero. Similarly, when $E = 0$, $\eta_E$ and the maximal eigenvalue of $\tilde{E}$ are zero. In these cases, we can simplify some of the bounds of Theorem \ref{thm:inexact_precon_bounds}. Some of the cubic polynomials that define the bounds will change in these cases. We will use a bar (e.g., $\bar{u}$) to denote cubic polynomials where $D = 0$, a hat (e.g., $\hat{u}$) to denote the polynomials where $E = 0$, and both (e.g., $\hat{\bar{u}}$) to denote both $D$ and $E$ being zero. We define the following cubic polynomials:
\begin{subequations}
	\begin{align}
		\bar{u}(\lambda) &=  \lambda^3 - \alpha_0 \lambda^2 - \left( \frac{\alpha_1 \alpha_2}{1+\eta_E} + \beta_0 \beta_1 \right)\lambda + \frac{\alpha_0 \alpha_1 \alpha_2}{1+\eta_E};
		\\
		\hat{u}(\lambda) &= \lambda^3 + (\beta_1 - \alpha_0)\lambda^2 - (\alpha_0 \beta_1 + \alpha_1 \alpha_2 + \beta_0 \beta_1  )\lambda + \alpha_0\alpha_1\alpha_2;
		\\
		\hat{\bar{u}}(\lambda) &= \lambda^3 - \alpha_0\lambda^2 - (\alpha_1 \alpha_2 + \beta_0 \beta_1)\lambda + \alpha_0\alpha_1\alpha_2;
		\\
		\hat{v}(\lambda) &= \lambda^3 - \beta_0 \lambda^2 - (\beta_0 \beta_1 + \beta_1 \beta_2)\lambda + \beta_0 \beta_1 \beta_2;
		\\
		\bar{w}(\lambda) &= \lambda^3 - \alpha_0 \lambda^2 - (\beta_0 \beta_1 + \beta_1 \beta_2)\lambda + \alpha_0 \beta_1 \beta_2.
	\end{align}
\end{subequations}
The following results then follow directly from Theorem \ref{thm:inexact_precon_bounds}; the proof is omitted.
\begin{corollary}
	\label{cor:inexact_precon_bounds_de0}
	\textbf{In the case $D = 0, \ E \ne 0$:} the eigenvalues of $\sM^{-1/2}\sK \sM^{-1/2}$ are bounded within the intervals
	\begin{equation} 
		\left[\bar{w}^-, \frac{\beta_0 - \sqrt{\beta_0^2 + 4\alpha_0 \alpha_1}}{2} \right] \ \bigcup \ \left[\bar{u}^+_{\min}, v^+_{\max} \right]. 
		\label{eq:bound_approx_d0}
	\end{equation}
	\textbf{In the case $D \ne 0, \ E =0$:} the eigenvalues are bounded in
	\begin{equation} 
		\left[w^-, \frac{\beta_0 - \sqrt{\beta_0^2 + \frac{4\alpha_0 \alpha_1}{1 + \eta_D}}}{2} \right] \ \bigcup \ \left[\hat{u}^+_{\min}, \hat{v}^+_{\max} \right]. 
		\label{eq:bound_approx_e0}
	\end{equation}
	\textbf{In the case $D = 0, \ E =0$:} the eigenvalues are bounded in
	\begin{equation} 
		\left[\bar{w}^-, \frac{\beta_0 - \sqrt{\beta_0^2 + 4\alpha_0 \alpha_1}}{2} \right] \ \bigcup \ \left[\hat{\bar{u}}^+_{\min}, \hat{v}^+_{\max} \right]. 
		\label{eq:bound_approx_de0}
	\end{equation}
\end{corollary}

\paragraph{Remark 3.}There is some looseness in the  bounds of Theorem \ref{thm:inexact_precon_bounds} when applied to the exact preconditioning case (i.e., $\alpha_i = \beta_i = 1$). This is a consequence of the fact that Theorem \ref{thm:inexact_precon_bounds} is based on the bounds for unpreconditioned matrices, which consider each matrix block individually, as opposed to the energy estimates approach in Section \ref{sec:prec}, which fully exploits the relationships between the blocks of the preconditioned matrix. For the extremal (lower negative and upper positive) bounds, the looseness is minor: Theorem \ref{thm:inexact_precon_bounds} gives a bound of 2 for the positive eigenvalues and approximately $-1.9$ on the negative eigenvalues, while we know from Theorem \ref{thm:bounds_prec} that tight bounds are approximately 1.8 and $-1.6$, respectively. For the interior bounds, however, we note that the bounds may be quite loose if $\eta_D$ or $\eta_E$ is very large. Fortunately, this is often not a concern in practical settings as having a large $\eta_D$ or $\eta_E$ generally implies that the spectral norm of $D$ or $E$ is large relative to that of $BA^{-1}B^T$ or $CS_1^{-1}C^T$, respectively. Often $D$ and/or $E$ are regularization terms, which tend to have a fairly small norm. Nonetheless, it should be acknowledged that the interior bounds may be pessimistic for some problems.

\paragraph{Remark 4.} The values of $\alpha_i$ and $\beta_i$ are rarely available in practical applications; they are typically coercivity constants or related quantities, which are proven to be independent of the mesh size but are not known explicitly. The values of $\beta_i$ should not typically generate a difficulty, and approximating them by $1$ or a value close to 1 should provide a reasonable approximation for the bounds. 
	
For the $\alpha_i$, let us provide some (partial) observations. For the solutions of the quadratic equations in Corollary~\ref{cor:inexact_precon_bounds_de0}, we can use the Taylor approximation $\sqrt{1+x} \lessapprox 1+\frac{x}{2}$
for $0< x \ll 1$ to conclude that 
$$  \frac{\beta_0 - \sqrt{\beta_0^2 + 4\alpha_0 \alpha_1}}{2} \gtrapprox - \frac{\alpha_0 \alpha_1}{\beta_0}. $$
This means that if $\alpha_0$ and $\alpha_1$ are small, then the above displayed expression would be a valid (albeit slightly less tight) upper negative bound, and if the $\alpha_i$ are uniformly bounded away from zero, then so is the bound. 

\section{Numerical experiments}
\label{sec:numex}
In this section we consider two slightly different variants of a Poisson control problem as in \cite{rees10b}. 
First, we consider a distributed control problem with Dirichlet boundary conditions:
\begin{subequations}
	\label{eqn:pdeco_dist_cont}
	\begin{align}
		\min_{u,f} \frac{1}{2}&||u-\hat{u}||_{L_2(\Omega)}^2 + \frac{\beta}{2} ||f||_{L_2(\Omega)}^2 \\
		\textrm{s.t.} \ \ -\nabla^2 u &= f \ \ \textrm{in } \Omega, \\
		u &= g \ \ \textrm{on } \partial \Omega,
	\end{align}
\end{subequations}
where $u$ is the state, $\hat{u}$ is the desired state, $0 < \beta \ll 1$ is a regularization parameter, $f$ is the control, and $\Omega$ is the domain with boundary $\partial \Omega$. Upon discretization, we obtain the system
\begin{equation}
	\label{eqn:pdeco_dist_discrete}
	\underbrace{
		\begin{bmatrix}
			M & K & 0 \\
			K & 0 & -M \\
			0 & -M & \beta M
	\end{bmatrix}}_{=: \sK}
	\begin{bmatrix}
		u_h \\
		\lambda \\
		f_h
	\end{bmatrix} = \begin{bmatrix}
		b \\
		d \\
		0
	\end{bmatrix},
\end{equation}
where $M$ is a symmetric positive definite mass matrix and $K$ is a symmetric positive definite discrete Laplacian. All blocks of $\sK$ are square (i.e., $n = m = p$).

As a second experiment we consider a boundary control problem:
\begin{subequations}
	\label{eqn:pdeco_dist_bnd}
	\begin{align}
		\min_{u,f} \frac{1}{2}&||u-\hat{u}||_{L_2(\Omega)}^2 + \frac{\beta}{2} ||g||_{L_2(\Omega)}^2 \\
		\textrm{s.t.} \ \ -\nabla^2 u &= 0 \ \ \textrm{in } \Omega, \\
		\frac{\partial u}{\partial n} &= g \ \ \textrm{on } \partial \Omega,
	\end{align}
\end{subequations}
which after discretization yields the linear system
\begin{equation}
	\label{eqn:pdeco_bnd_discrete}
	\underbrace{
		\begin{bmatrix}
			M & K & 0 \\
			K & 0 & -E \\
			0 & -E^T & \beta M_b
	\end{bmatrix}}_{=: \sK_{\partial}}
	\begin{bmatrix}
		u_h \\
		\lambda \\
		g_h
	\end{bmatrix} = \begin{bmatrix}
		b_{\partial} \\
		d_{\partial} \\
		0
	\end{bmatrix},
\end{equation}
where $M_b \in \mathbb{R}^{n_b \times n_b}$ (with $n_b < n$) is a boundary mass matrix. Thus, in the distributed control problem the mass matrix in the (2,3)/(3,2)-block is square and in the boundary control problem we consider a rectangular version of it. 

In all experiments that follow we set $\Omega$ to be the unit square.  We use uniform \textbf{Q1} finite elements and set $\beta = 10^{-3}$. The MATLAB code of Rees \cite{rees10b_code} was used to generate the linear systems.

\subsection{Eigenvalues of unpreconditioned matrices}
\label{sec:numex_unprec}
Here we compare the eigenvalues of $\sK$ \eqref{eqn:pdeco_dist_discrete} and $\sK_{\partial}$ \eqref{eqn:pdeco_bnd_discrete} to the eigenvalue bounds predicted by Theorem \ref{thm:bounds_unprec}. We use MATLAB's \texttt{eigs}/\texttt{svds} functions to compute the minimum and maximum eigenvalues/singular values of the matrix blocks $M, K, M_b$, and $E$.

We note that for the distributed control matrix $\sK$, all blocks are square and the (3,3)-block is positive definite; therefore, we can also use our results to obtain bounds on the re-ordered matrix
\begin{equation}
	\label{eqn:kflip}
	\sK_{\textrm{flip}} = \begin{bmatrix}
		\beta M & -M & 0 \\
		-M & 0 & K \\
		0 & K & M
	\end{bmatrix}.
\end{equation}
Both orderings generate the same extremal bounds but different interior bounds.

\begin{figure}[tbh!]
	\begin{subfigure}{.49\textwidth}
		\centering
		\includegraphics[width=.8\linewidth]{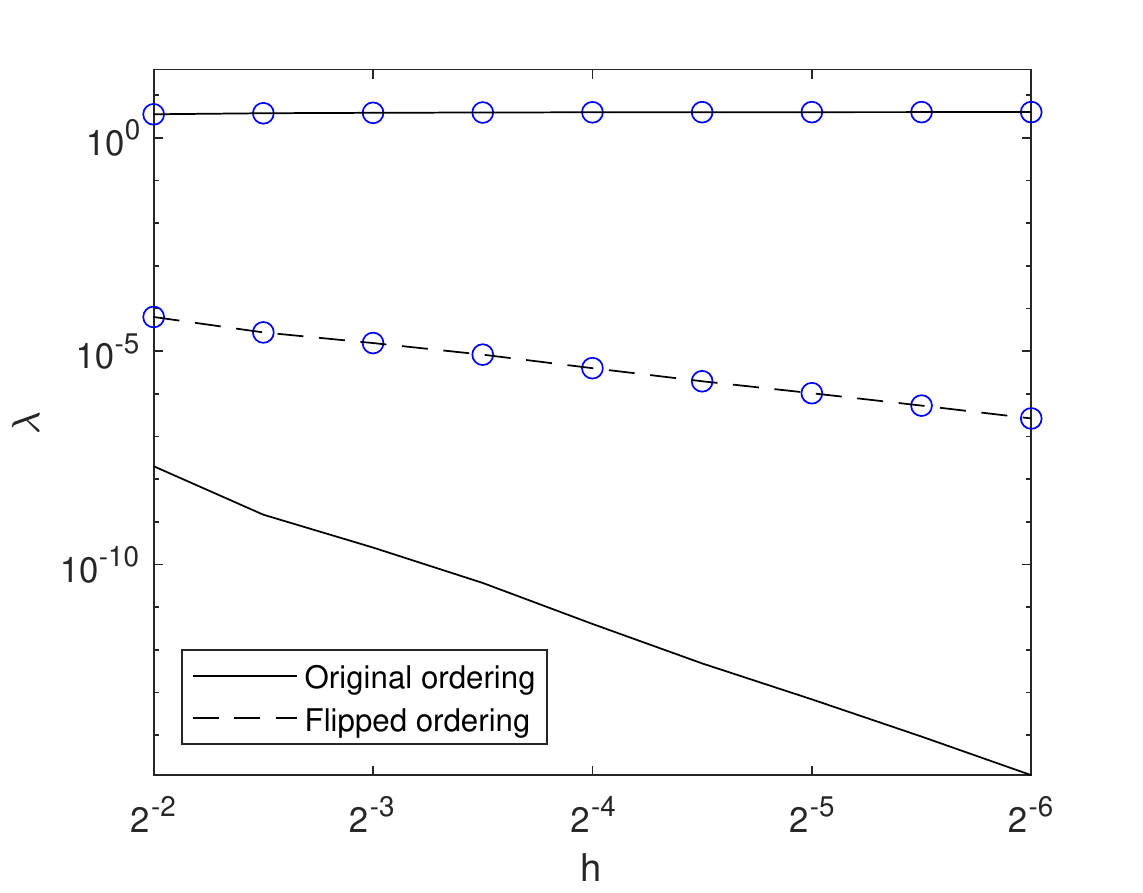}  
		\caption{Positive eigenvalues of $\sK$}
		\label{fig:pos_dist}
	\end{subfigure}
	\begin{subfigure}{.49\textwidth}
		\centering
		\includegraphics[width=.8\linewidth]{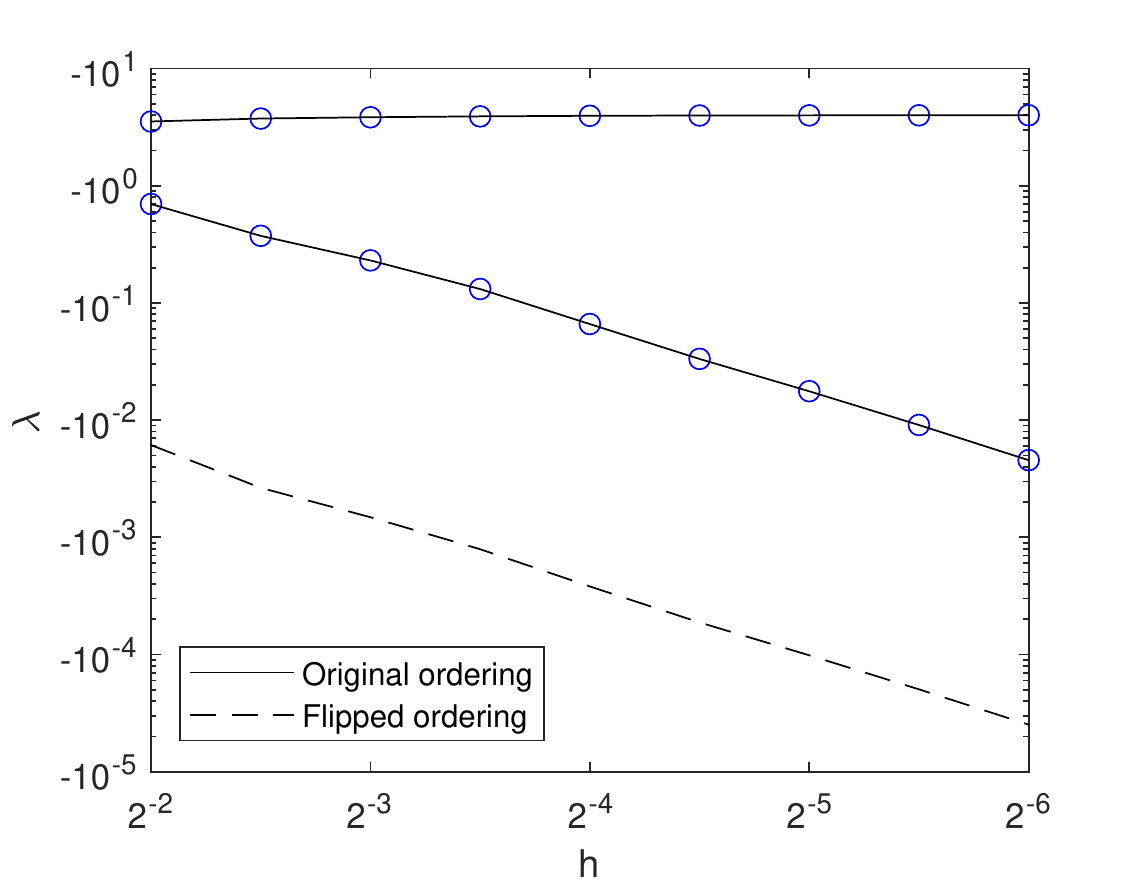}  
		\caption{Negative eigenvalues of $\sK$}
		\label{fig:neg_dist}
	\end{subfigure}
	\newline
	\begin{subfigure}{.49\textwidth}
		\centering
		\includegraphics[width=.8\linewidth]{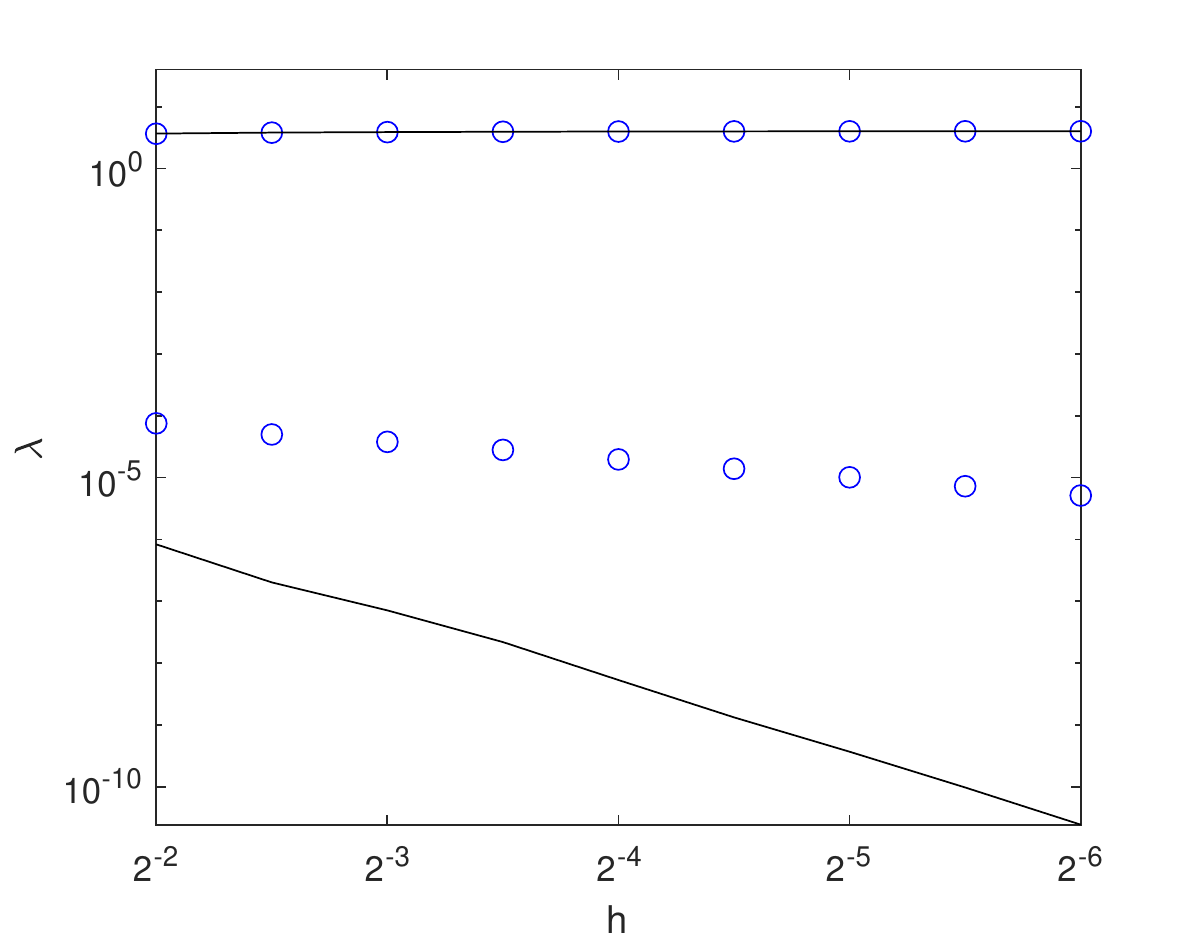}  
		\caption{Positive eigenvalues of $\sK_{\partial}$}
		\label{fig:pos_bnd}
	\end{subfigure}
	\begin{subfigure}{.49\textwidth}
		\centering
		\includegraphics[width=.8\linewidth]{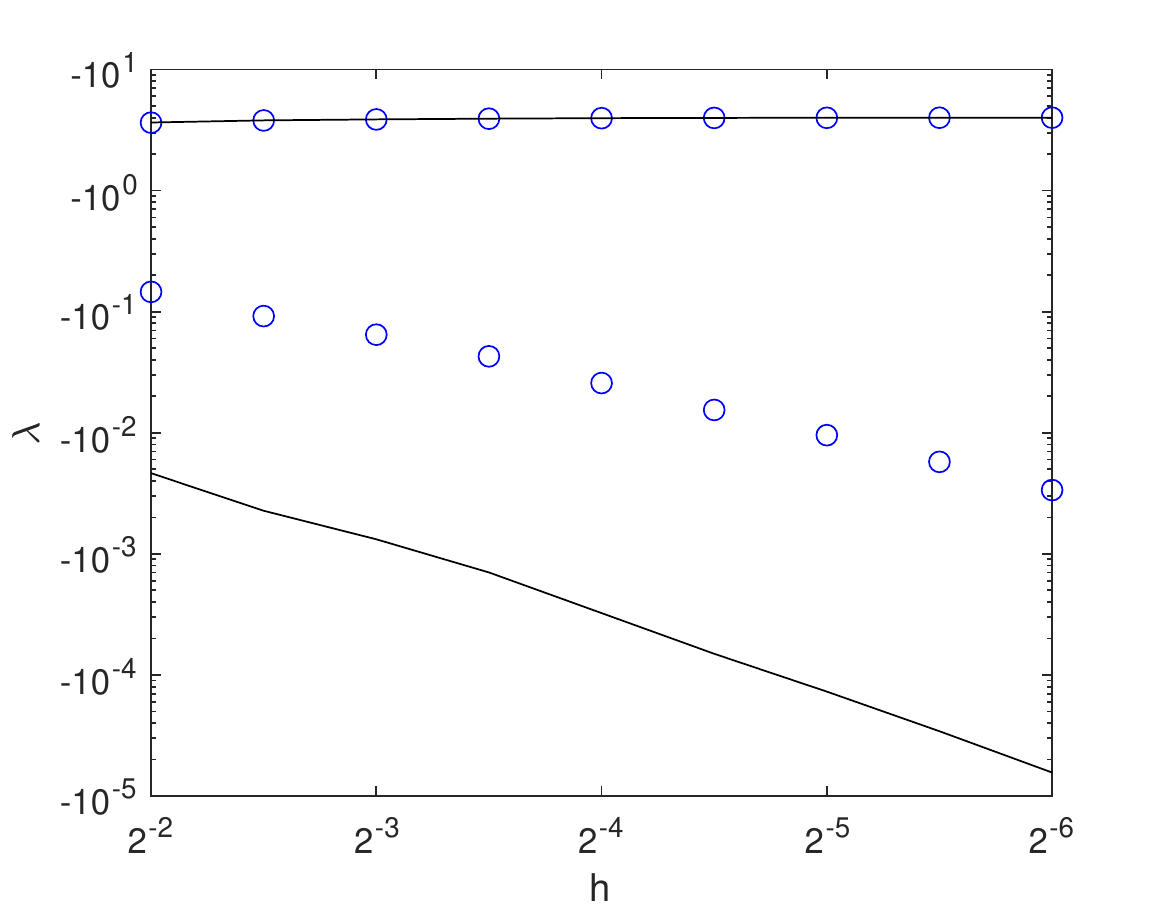}  
		\caption{Negative eigenvalues of $\sK_{\partial}$}
		\label{fig:neg_bnd}
	\end{subfigure}
	\caption{Largest and smallest positive (left) and negative (right) eigenvalues of the distributed control matrix $\sK$ (top) and the boundary control matrix $\sK_{\partial}$ (bottom). Blue circles indicate the eigenvalues, and black lines the bounds given by Theorem \ref{thm:bounds_unprec}. For $\sK$, the dashed lines indicate the bounds obtained by applying Theorem \ref{thm:bounds_unprec} to the reordered matrix $\sK_{\textrm{flip}}$.}
	\label{fig:pdeco_eig_bounds}
\end{figure}

Comparisons of the predicted eigenvalue bounds to the actual eigenvalues are shown in Section \ref{fig:pdeco_eig_bounds}. In the distributed control case, we show the bounds obtained for both the original matrix $\sK$ and those for the reordered matrix $\sK_{\textrm{flip}}$. In all cases, the bounds for the extremal eigenvalues are quite tight: this is because the $K$ block has the largest eigenvalues ($O(1)$ compared to $O(h^2)$ for the others -- see \cite[Proposition 1.29 and Theorem 1.32]{elman05}). Therefore, the largest positive and smallest negative eigenvalue of both $\sK$ and $\sK_{\partial}$ tend towards $\mu_{\max}^K$ and $-\mu_{\max}^K$, respectively. Referring to the proofs of the extremal bounds in Theorem \ref{thm:bounds_unprec}, the $3 \times 3$ matrices $R$ will contain a $\mu_{\max}^K$ (or $-\mu_{\max}^K$) term, with other lower-order terms. This means that the extremal eigenvalues of $R$ (and therefore the predicted eigenvalue bounds) will also be close to $\pm \mu_{\max}^K$.

It is more difficult to capture the interior bounds. In the distributed control case, each ordering (either the original ordering with $M$ in the leading block or the flipped ordering with $\beta M$ in the leading block) gives one bound that is quite tight and one that is loose. In the boundary control case, both interior bounds are loose.

\subsection{Eigenvalues of preconditioned matrices}
\label{sec:numex_prec}
We now consider preconditioning strategies for PDE-constrained optimization. We examine eigenvalue bounds with both exact and approximate Schur complements. 

For the distributed control problem, we work with the reordered matrix $\sK_{\textrm{flip}}$ \eqref{eqn:kflip}. The Schur complement preconditioner for $\sK_{\textrm{flip}}$ is
\begin{equation}
	\label{eqn:dist_ideal}
	\sM= \begin{bmatrix}
		\beta M & 0 & 0 \\
		0 & \frac{1}{\beta}M & 0 \\
		0 & 0 & M + \beta KM^{-1} K
	\end{bmatrix}.
\end{equation}
The first and second blocks are mass matrices, which are cheap to invert, so we leave these terms as they are. For the second Schur complement $S_2 = M + \beta KM^{-1}K$, we  use the approximation $\tilde{S}_2 = \left( M + \sqrt{\beta}K \right)M^{-1} \left( M + \sqrt{\beta}K \right)$ proposed by Pearson and Wathen \cite{pearson12} to obtain the preconditioner:
\begin{equation}
	\label{eqn:dist_approx}
	\tilde{\sM} = \begin{bmatrix}
		\beta M & 0 & 0 \\
		0 & \frac{1}{\beta}M & 0 \\
		0 & 0 & \left( M + \sqrt{\beta}K \right)M^{-1} \left( M + \sqrt{\beta}K \right) 
	\end{bmatrix}.
\end{equation}
Per \cite[Theorem 4]{pearson12}, the eigenvalues of $\tilde{S}_2^{-1}S_2$ satisfy $\Lambda(\tilde{S}_2^{-1}S_2) \in \left[ \frac{1}{2}, 1 \right]$. Thus $\tilde{\sM}$ satisfies Theorem \ref{thm:spec_equiv} with $\alpha_0 = \beta_0 = \alpha_1 = \beta_1 = 1$, $\alpha_2 = \frac{1}{2}$, and $\beta_2 = 1$.

\begin{figure}[tbh!]
	\begin{subfigure}{.49\textwidth}
		\centering
		\includegraphics[width=.8\linewidth]{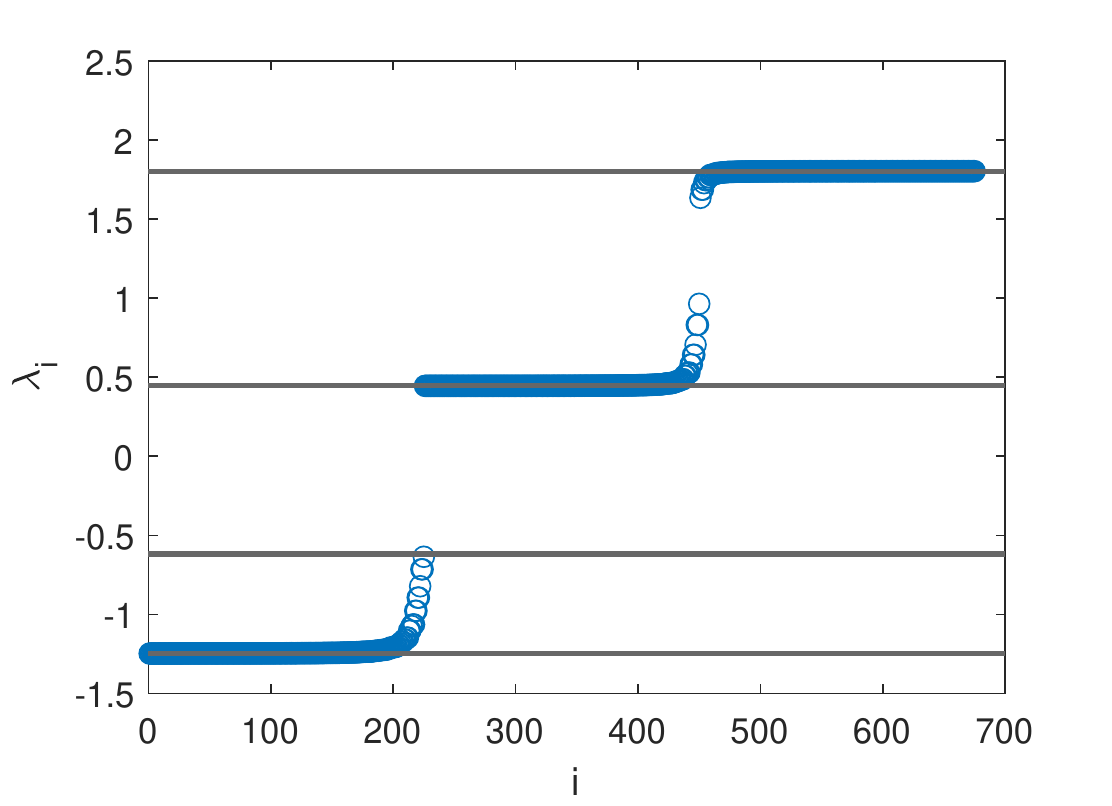}  
		\caption{$\sM^{-1}\sK_{\textrm{flip}}$}
		\label{fig:sc_exact}
	\end{subfigure}
	\begin{subfigure}{.49\textwidth}
		\centering
		\includegraphics[width=.8\linewidth]{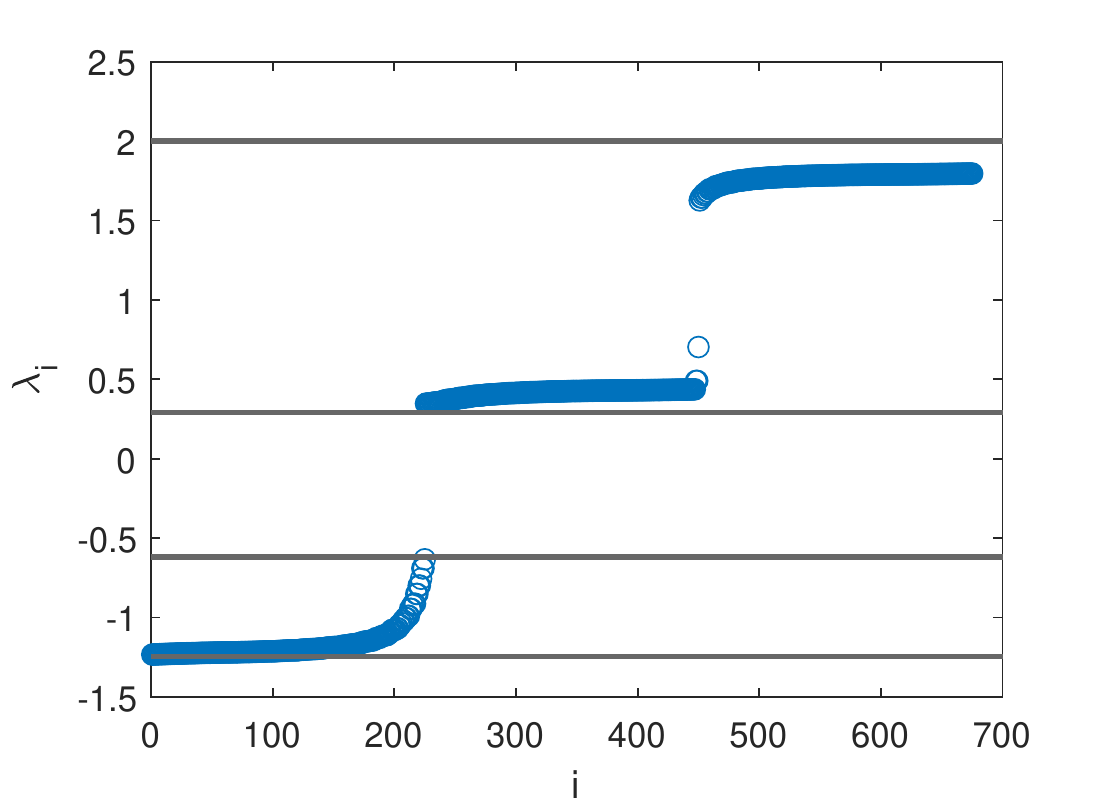}  
		\caption{$\tilde{\sM}^{-1}\sK_{\textrm{flip}}$}
		\label{fig:sc_inexact}
	\end{subfigure}
	\caption{Eigenvalue plots for reordered distributed Poisson control matrix $\sK_{\textrm{flip}}$, with exact preconditioner $\sM$ (left) and approximate preconditioner $\tilde{\sM}$ (right). Eigenvalues are shown by the blue circles; eigenvalue bounds (from Theorem \ref{thm:bounds_prec_d0} on left and Corollary \ref{cor:inexact_precon_bounds_de0} on right) are shown by lines.}
	\label{fig:pdeco_dist_precon}
\end{figure}

Plots of the preconditioned eigenvalues for $h = 2^{-4}$ are shown in Section \ref{fig:pdeco_dist_precon}. The value $\eta_E$, defined as the maximal eigenvalue of
$$
(CS_1^{-1}C^T)^{-1}E = \beta(KM^{-1}K)^{-1}M,
$$
is approximately $2.6 \times 10^{-7}$. We note that for 2D problems with uniform \textbf{Q1} finite element discretizations the value $\eta_E$ is $O(\beta h^4)$, and will thus be small in general.

Comparing Figures \ref{fig:sc_exact} and \ref{fig:sc_inexact}, we notice that the bounds on the negative eigenvalue bounds do not change when we use the approximate Schur complement, but the lower positive bound becomes smaller (from 0.4450 with the exact Schur complement to 0.2929 for the approximate Schur complement) and the upper positive bound becomes larger (1.8019 in the exact case and 2 in the approximate case). We note that the eigenvalues appear to be very close to the predicted bounds except for the upper positive eigenvalues of $\tilde{\sM}^{-1}\sK$. As discussed in Remark 3, this kind of minor looseness in the upper bound can happen when we have highly accurate Schur complement approximations (as in this case, where we are exactly inverting the two $M$ blocks).

For the matrix $\sK_{\partial}$ arising from the boundary control problem, the Schur complement preconditioner is
\begin{equation}
	\label{eqn:bnd_ideal}
	\sM_{\partial} = \begin{bmatrix}
		M & 0 & 0 \\
		0 & KM^{-1}K & 0 \\
		0 & 0 & \beta M_b + E^T\left( KM^{-1}K \right)^{-1}E
	\end{bmatrix}.
\end{equation}
In practice, the first Schur complement $KM^{-1}K$ can be inverted approximately with, for example, a multigrid method. For the second Schur complement, we note that for the 2D Poisson control problem on a uniform \textbf{Q1} grid, the eigenvalues of $K$ are between $O(h^2)$ and $O(1)$, while those of the mass matrices are all $O(h^2)$. Therefore, the term $\beta M_b$ will dominate $E^T\left( KM^{-1}K \right)^{-1}E$ for all but very small values of $\beta$ (which are not commonly used in practice). Therefore, an approximate preconditioner for $\sK_{\partial}$ is given by
\begin{equation}
	\label{eqn:bnd_approx}
	\tilde{\sM}_{\partial} = \begin{bmatrix}
		M & 0 & 0 \\
		0 & KM^{-1}K & 0 \\
		0 & 0 & \beta M_b
	\end{bmatrix}.
\end{equation}
We note that this preconditioner is presented in \cite{rees10b}, though it is derived there from the block-$2 \times 2$ formulation of $\sK$.

\begin{figure}[tbh!]
	\begin{subfigure}{.49\textwidth}
		\centering
		\includegraphics[width=.8\linewidth]{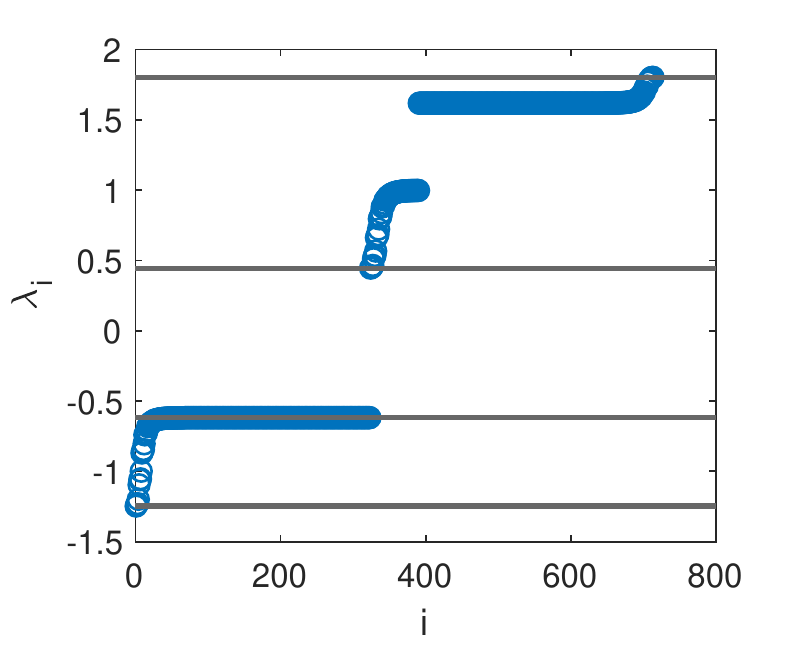}  
		\caption{$\sM_{\partial}^{-1}\sK_{\partial}$}
		\label{fig:bnd_sc_exact}
	\end{subfigure}
	\begin{subfigure}{.49\textwidth}
		\centering
		\includegraphics[width=.8\linewidth]{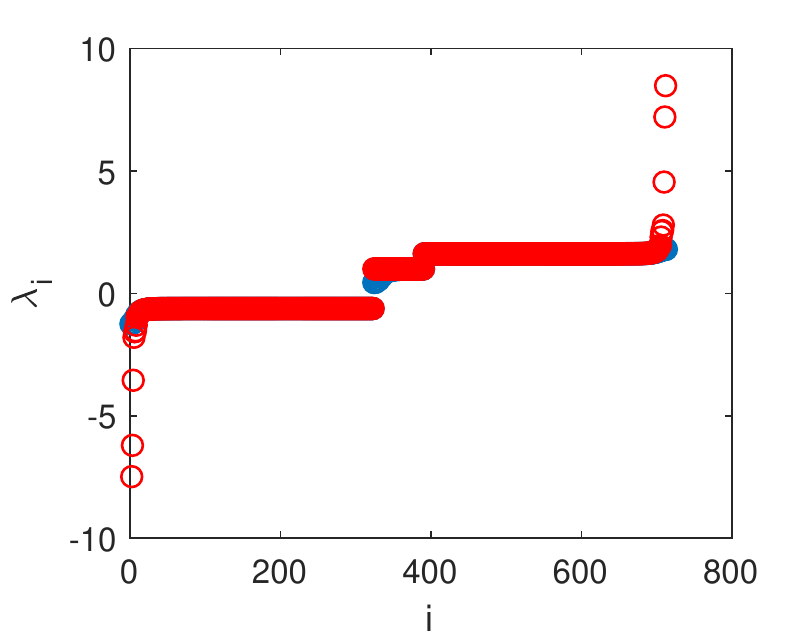}  
		\caption{$\sM_{\partial}^{-1}\sK_{\partial}$ (blue); $\tilde{\sM}_{\partial}^{-1}\sK_{_{\partial}}$ (red)}
		\label{fig:bnd_sc_overlay}
	\end{subfigure}
	\caption{Eigenvalue plots for the boundary control problem $\sK_{\partial}$, with exact preconditioner $\sM$  and approximate preconditioner $\tilde{\sM}$. 
		On left, eigenvalue bounds (from Theorem \ref{thm:bounds_prec_d0}) are shown by horizontal lines. Right: all but the single smallest and largest eigenvalues for both preconditioners; values for $\sM$ are in blue, and those for $\tilde{\sM}$ are in red.}
	\label{fig:pdeco_bnd_precon}
\end{figure}

Plots of the preconditioned eigenvalues for $h = 2^{-4}$ are shown in Figure \ref{fig:pdeco_bnd_precon}. For $\tilde{\sM}_{\partial}^{-1}\sK_{_{\partial}}$, we note that there is a single large-magnitude positive eigenvalue and a large-magnitude negative eigenvalue, whose absolute values are nearly 200, and they make it difficult to see how the other eigenvalues compare to to those of $\sM_{\partial}^{-1}\sK_{_{\partial}}$. Thus in Figure \ref{fig:bnd_sc_overlay}, we omit the largest and smallest eigenvalues of $\tilde{\sM}_{\partial}^{-1}\sK_{_{\partial}}$ and overlay the others on the corresponding eigenvalues of $\sM_{\partial}^{-1}\sK_{_{\partial}}$. We notice that most other eigenvalues of $\tilde{\sM}_{\partial}^{-1}\sK_{_{\partial}}$ remain close to those of $\sM_{\partial}^{-1}\sK_{_{\partial}}$. We also note that the performance of this preconditioner depends on $\beta$: in particular, the performance of the Schur complement approximation deteriorates for small $\beta$. We refer to \cite{pearson12, rees10b} for further discussion of this.

It is evident from Figure \ref{fig:bnd_sc_exact} that our bounds are tight and effective. Unlike in the boundary control example, the Schur complement approximation $\tilde{S}_2$ used here does not have $\beta$- and $h$-independent constants $\alpha_2, \beta_2$ such that $\Lambda(\tilde{S}_2^{-1}S_2) \in [\alpha_2, \beta_2]$, so our analyses on Schur complement approximations in Section Section \ref{sec:prec_inexact} are difficult to apply. Nonetheless, we observe from Figure \ref{fig:bnd_sc_overlay} that most of the eigenvalues of $\tilde{\sM}_{\partial}^{-1}\sK_{\partial}$ remain very close to those of $\sM_{\partial}^{-1}\sK_{\partial}$. Thus, we see that the eigenvalue bounds for the ``ideal'' Schur complement preconditioner may still be of use, provided that we have an effective approximation of the Schur complement.

\section{Conclusions}
\label{sec:conclusions}
The increasing importance of double saddle-point systems requires attention to spectral properties of the matrices involved. We have shown that energy estimates are an effective tool for obtaining eigenvalue bounds in this case. The assumptions we make are rather general, and the analysis covers a large class of problems. 

There are several directions for potential future work. Following Remark \ref{rem:looseness}, specific assumptions on the magnitudes of the norms of the matrices $D$ and $E$ may yield additional results and insights on the bounds. The rank structure of the blocks may also have a significant effect, and it may be useful to eliminate the positive definiteness requirement of $A$ and/or consider rank-deficient $B$ and $C$. It may also be useful to consider nonsymmetric double saddle-point systems, and block triangular rather than block diagonal preconditioners, although when symmetry is lost eigenvalue bounds may not be a sufficient tool for predicting convergence rates of Krylov subspace solvers.

The tightness of our bounds indicates that spectral analysis is beneficial in capturing the properties of the matrices involved. This, in turn, makes it possible to  effectively predict the convergence rate of Krylov subspace solvers for this important class of problems.

\bibliographystyle{abbrv}
\bibliography{bg21-ref}

\end{document}